\documentclass[11pt]{amsart}

\usepackage{epigamath}

\usepackage[english]{babel}

\numberwithin{equation}{section}

\usepackage[shortlabels]{enumitem}
\setlist[enumerate,1]{label={\rm(\roman*)}, ref={\rm\roman*}} 

\newlist{a-enumerate}{enumerate}{2}
\setlist[a-enumerate,1]{label={\rm(\alph*)}, ref={\rm\alph*}}

\usepackage{amsfonts}
\usepackage{mathrsfs}
\usepackage{amscd}
\usepackage{amsmath}
\usepackage{amssymb}
\usepackage{latexsym}
\usepackage{lscape}
\usepackage{xypic}
\usepackage{comment}
\usepackage{amscd}
\usepackage{tikz} 
\usetikzlibrary{matrix,arrows}
\usepackage[nice]{nicefrac}
\usepackage{dsfont}

\newtheorem{thm}{Theorem}[section]

\newtheorem{prop}[thm]{Proposition}
\newtheorem{lem}[thm]{Lemma}
\newtheorem{cor}[thm]{Corollary}

\theoremstyle{definition}
\newtheorem{dfn}[thm]{Definition}

\theoremstyle{remark}
\newtheorem{rmk}[thm]{Remark}

\newcommand{\nc}{\newcommand}

\nc{\on}{\operatorname}
\nc{\fraka}{{\mathfrak a}} \nc{\bba}{{\mathbf a}}
\nc{\frakb}{{\mathfrak b}}
\nc{\frakc}{{\mathfrak c}}
\nc{\frakd}{{\mathfrak d}}
\nc{\frake}{{\mathfrak e}}
\nc{\frakf}{{\mathfrak f}}
\nc{\frakg}{{\mathfrak g}}
\nc{\frakh}{{\mathfrak h}}
\nc{\fraki}{{\mathfrak i}}
\nc{\frakj}{{\mathfrak j}}
\nc{\frakk}{{\mathfrak k}}
\nc{\frakl}{{\mathfrak l}}
\nc{\frakm}{{\mathfrak m}}
\nc{\frakn}{{\mathfrak n}}
\nc{\frako}{{\mathfrak o}}
\nc{\frakp}{{\mathfrak p}}
\nc{\frakq}{{\mathfrak q}}
\nc{\frakr}{{\mathfrak r}}
\nc{\fraks}{{\mathfrak s}}
\nc{\frakt}{{\mathfrak t}}
\nc{\fraku}{{\mathfrak u}}
\nc{\frakv}{{\mathfrak v}}
\nc{\frakw}{{\mathfrak w}}
\nc{\frakx}{{\mathfrak x}}
\nc{\fraky}{{\mathfrak y}}
\nc{\frakz}{{\mathfrak z}}
\nc{\frakA}{{\mathfrak A}}
\nc{\frakB}{{\mathfrak B}}
\nc{\frakC}{{\mathfrak C}}
\nc{\frakD}{{\mathfrak D}}
\nc{\frakE}{{\mathfrak E}}
\nc{\frakF}{{\mathfrak F}}
\nc{\frakG}{{\mathfrak G}}
\nc{\frakH}{{\mathfrak H}}
\nc{\frakI}{{\mathfrak I}}
\nc{\frakJ}{{\mathfrak J}}
\nc{\frakK}{{\mathfrak K}}
\nc{\frakL}{{\mathfrak L}}
\nc{\frakM}{{\mathfrak M}}
\nc{\frakN}{{\mathfrak N}}
\nc{\frakO}{{\mathfrak O}}
\nc{\frakP}{{\mathfrak P}}
\nc{\frakQ}{{\mathfrak Q}}
\nc{\frakR}{{\mathfrak R}}
\nc{\frakS}{{\mathfrak S}}
\nc{\frakT}{{\mathfrak T}}
\nc{\frakU}{{\mathfrak U}}
\nc{\frakV}{{\mathfrak V}}
\nc{\frakW}{{\mathfrak W}}
\nc{\frakX}{{\mathfrak X}}
\nc{\frakY}{{\mathfrak Y}}
\nc{\frakZ}{{\mathfrak Z}}
\nc{\bbA}{{\mathbb A}}
\nc{\bbB}{{\mathbb B}}
\nc{\bbC}{{\mathbb C}}
\nc{\bbD}{{\mathbb D}}
\nc{\bbE}{{\mathbb E}}
\nc{\bbF}{{\mathbb F}} \nc{\bbf}{{\mathbf f}}
\nc{\bbG}{{\mathbb G}}
\nc{\bbH}{{\mathbb H}}
\nc{\bbI}{{\mathbb I}}
\nc{\bbJ}{{\mathbb J}}
\nc{\bbK}{{\mathbb K}}
\nc{\bbL}{{\mathbb L}}
\nc{\bbM}{{\mathbb M}}
\nc{\bbN}{{\mathbb N}}
\nc{\bbO}{{\mathbb O}}
\nc{\bbP}{{\mathbb P}}
\nc{\bbQ}{{\mathbb Q}}
\nc{\bbR}{{\mathbb R}}
\nc{\bbS}{{\mathbb S}}
\nc{\bbT}{{\mathbb T}}
\nc{\bbU}{{\mathbb U}}
\nc{\bbV}{{\mathbb V}}
\nc{\bbW}{{\mathbb W}}
\nc{\bbX}{{\mathbb X}}
\nc{\bbY}{{\mathbb Y}}
\nc{\bbZ}{{\mathbb Z}}
\nc{\calA}{{\mathcal A}}
\nc{\calB}{{\mathcal B}}
\nc{\calC}{{\mathcal C}}
\nc{\calD}{{\mathcal D}}
\nc{\calE}{{\mathcal E}}
\nc{\calF}{{\mathcal F}}
\nc{\calG}{{\mathcal G}}
\nc{\calH}{{\mathcal H}}
\nc{\calI}{{\mathcal I}}
\nc{\calJ}{{\mathcal J}}
\nc{\calK}{{\mathcal K}}
\nc{\calL}{{\mathcal L}}
\nc{\calM}{{\mathcal M}}
\nc{\calN}{{\mathcal N}}
\nc{\calO}{{\mathcal O}}
\nc{\calP}{{\mathcal P}}
\nc{\calQ}{{\mathcal Q}}
\nc{\calR}{{\mathcal R}}
\nc{\calS}{{\mathcal S}}
\nc{\calT}{{\mathcal T}}
\nc{\calU}{{\mathcal U}}
\nc{\calV}{{\mathcal V}}
\nc{\calW}{{\mathcal W}}
\nc{\calX}{{\mathcal X}}
\nc{\calY}{{\mathcal Y}}
\nc{\calZ}{{\mathcal Z}}

\nc{\scrA}{{\mathscr A}}
\nc{\scrB}{{\mathscr B}}
\nc{\scrR}{{\mathscr R}}

\nc{\bnu}{{\bar{ \nu}}}

\nc{\olO}{\bar{\calO}}

\nc{\al}{{\alpha}} 
\nc{\be}{{\beta}}
\nc{\ga}{{\gamma}} \nc{\Ga}{{\Gamma}}
 \nc{\hGa}{\hat{\Gamma}}
\nc{\ve}{{\varepsilon}} 
\nc{\la}{{\lambda}} \nc{\La}{{\Lambda}}
\nc{\om}{\omega} \nc{\Om}{\Omega} 
\nc{\sig}{{\sigma}} \nc{\Sig}{{\Sigma}}

\nc{\tnb}{\psi_{\rm tame}}
\nc{\oM}{\overline{{M}}}
\nc{\op}{{\on{op}}}
\nc{\ad}{{\on{ad}}}
\nc{\alg}{{\on{alg}}}
\nc{\Ad}{{\on{Ad}}}
\nc{\Adm}{{\on{Adm}}}
\nc{\Aut}{{\on{Aut}}}
\nc{\Bun}{{\on{Bun}}}
\nc{\cha}{{\on{char}}}
\nc{\der}{{\on{der}}}
\nc{\Der}{{\on{Der}}}
\nc{\diag}{{\on{diag}}}
\nc{\End}{{\on{End}}}
\nc{\Fl}{{\calF\!\ell}}
\nc{\Tr}{{\on{Transp}}}
\nc{\TR}{{\calT\!\calR}}
\nc{\Gal}{{\on{Gal}}}
\nc{\Gr}{{\on{Gr}}}
\nc{\rH}{{\on{H}}}
\nc{\Hom}{{\on{Hom}}}
\nc{\IC}{{\on{IC}}}
\nc{\id}{{\on{id}}}
\nc{\Id}{{\on{Id}}}
\nc{\ind}{{\on{ind}}}
\nc{\Ind}{{\on{Ind}}}
\nc{\Lie}{{\on{Lie}}}
\nc{\Pic}{{\on{Pic}}}
\nc{\pr}{{\on{pr}}}
\nc{\Res}{{\on{Res}}}
\nc{\res}{{\on{res}}}
\nc{\Sat}{{\on{Sat}}}
\nc{\s}{{\on{sc}}}
\nc{\drv}{{\mathrm{der}}}
\nc{\sgn}{{\on{sgn}}}
\nc{\Spec}{{\on{Spec}}}
\nc{\Spf}{\on{Spf}} 
\nc{\Sph}{\on{Sph}}
\nc{\St}{{\on{St}}}
\nc{\tr}{{\on{tr}}}
\nc{\Mod}{{\mathrm{-Mod}}}
\nc{\Hilb}{{\on{Hilb}}} 
\nc{\Ext}{{\on{Ext}}} 
\nc{\vs}{{\on{Vec}}}
\nc{\ev}{{\on{ev}}}
\nc{\nO}{{\breve{\calO}}}
\nc{\tS}{{\tilde{S}}}
\nc{\spe}{{\on{sp}}}
\nc{\loc}{{\on{loc}}}

\nc{\nscrR}{{\mathscr{R}^{\on{nr}}}}

\nc{\GL}{{\on{GL}}}
\nc{\U}{{\on{U}}}
\nc{\Gl}{\on{Gl}} 
\nc{\GSp}{{\on{GSp}}}
\nc{\gl}{{\frakg\frakl}}
\nc{\SL}{{\on{SL}}} 
\nc{\SU}{{\on{SU}}} 
\nc{\SO}{{\on{SO}}}
\nc{\PGL}{{\on{PGL}}}

\nc{\Conv}{{\on{Conv}}}
\nc{\Rep}{{\on{Rep}}}
\nc{\Dom}{{\on{Dom}}}
\nc{\act}{{\on{act}}}
\nc{\nr}{{\on{nr}}}
\nc{\ctf}{{\on{ctf}}}

\nc{\str}{{\on{-}}} 
\nc{\os}{{\bar{s}}}
\nc{\oeta}{{\bar{\eta}}}

\nc{\hookto}{\hookrightarrow}
\nc{\longto}{\longrightarrow}
\nc{\leftto}{\leftarrow}
\nc{\onto}{\twoheadrightarrow}
\nc{\lonto}{\twoheadleftarrow}

\nc{\uG}{{\underline{G}}}
\nc{\uA}{{\underline{A}}}
\nc{\uS}{{\underline{S}}}
\nc{\uT}{{\underline{T}}}
\nc{\uM}{{\underline{M}}}
\nc{\uP}{{\underline{P}}}
\nc{\uB}{{\underline{B}}}
\nc{\uN}{{\underline{N}}}

\nc{\ucG}{{\underline{\calG}}}
\nc{\ucA}{{\underline{\calA}}}
\nc{\ucS}{{\underline{\calS}}}
\nc{\ucT}{{\underline{\calT}}}
\nc{\ucM}{{\underline{\calM}}}
\nc{\ucP}{{\underline{\calP}}}
\nc{\ucN}{{\underline{\calN}}}

\nc{\bF}{{\breve{F}}}

\nc{\oFl}{{\overline{\Fl}}} 
\nc{\bU}{{\overline{U}}}
\nc{\tGr}{{\tilde{\Gr}}}
\nc{\cGr}{\calG\! r}
\nc{\oGr}{\overline{\on{Gr}}} 
\nc{\ocGr}{\overline{\calG\! r}}
\nc{\co}{{\colon}}
\nc{\sch}[1]{(Sch/{#1})}
\nc{\HypLoc}[1]{HypLoc({#1})}

\nc{\ohtimes}{\stackrel{!}{\otimes}}
\nc{\boxtilde}{\widetilde{\boxtimes}}
\nc{\vstar}{{\varhexstar}}

\nc{\Div}{\on{Div}}

\nc{\bslash}{\backslash}
\nc{\algQl}{{\bar{\bbQ}_\ell}}
\nc{\sF}{{\bar{F}}}
\nc{\nF}{{\breve{F}}}
\nc{\nW}{{W^{\on{nr}}}}
\nc{\sk}{{\bar{k}}}
\nc{\cont}{\on{c}}
\nc{\Supp}{\on{Supp}}
\nc{\blt}{\bullet}  
\nc{\dom}{\on{dom}}
\nc{\scon}{{\on{sc}}} 
\nc{\Affine}{\on{Aff}} 
\nc{\nscrA}{\mathscr{A}^{\on{nr}}} 
\nc{\nfraka}{{\bbf^{\on{nr}}}}
\nc{\ran}{{\rangle}}
\nc{\lan}{{\langle}}
\nc{\bk}{{\bar{k}}}
\nc{\tF}{{\tilde{F}}}
\nc{\sS}{{\bar{S}}}
\nc{\LG}{{^\text{L}\hspace{-0.04cm}G}}
\nc{\LL}{{^\text{L}\hspace{-0.07cm}L}}

\nc{\pot}[1]{ [\hspace{-0,5mm}[ {#1} ]\hspace{-0,5mm}] }
\nc{\rpot}[1]{ (\hspace{-0,7mm}( {#1} )\hspace{-0,7mm}) }

\nc{\defined}{\hspace{0.1cm}\stackrel{\text{\tiny \rm def}}{=}\hspace{0.1cm}}

\renewcommand{\Im}{\on{Im}}
\nc{\aff}{\mathrm{aff}}
\nc{\red}{\mathrm{red}}
\nc{\sep}{\mathrm{sep}}
\nc{\Flag}{\on{Fl}}
\nc{\et}{\mathrm{\acute{e}t}}

\nc\lowsim{\vbox to 0pt{\vss\hbox{$\scriptstyle\sim$}\vskip-2pt}}
\nc{\longdash}{\,\rule[2.5pt]{12pt}{0.3pt}\,}

\EpigaVolumeYear{9}{2025} \EpigaArticleNr{9} \ReceivedOn{September 30, 2023}
\InFinalFormOn{September 10, 2024}
\AcceptedOn{September 24, 2024}

\title{Cellular pavings of fibers of convolution morphisms}
\titlemark{Cellular pavings of fibers of convolution morphisms}

\author{Thomas J.~Haines}
\address{Department of Mathematics, University of Maryland, College Park, MD 20742-4015, USA}
\email{tjh@math.umd.edu}

\authormark{T.\,J.~Haines}

\AbstractInEnglish{This article proves, in the case of split groups over arbitrary fields, that all fibers of convolution morphisms attached to parahoric affine flag varieties are paved by products of affine lines and affine lines minus a point. This applies in particular to the affine Grassmannian and to the convolution morphisms in the context of the geometric Satake correspondence. The second part of the article extends these results over $\mathbb Z$. Those in turn relate to the recent work of Cass--van den Hove--Scholbach on the geometric Satake equivalence for integral motives, and provide some alternative proofs for some of their results.}

\MSCclass{22E67, 20G25, 22E57}
 
\KeyWords{Loop groups and parahoric groups, affine flag varieties, affine Grassmannians}

\acknowledgement{Research partially supported by NSF DMS-2200873. Disclaimer: Any opinions, findings, and conclusions or recommendations expressed in this material are those of the author and do not necessarily reflect the views of the National Science Foundation.}

\begin{document}


\maketitle

\begin{prelims}

\DisplayAbstractInEnglish

\bigskip

\DisplayKeyWords

\medskip

\DisplayMSCclass

\end{prelims}


\newpage

\setcounter{tocdepth}{1}

\tableofcontents


\section{Introduction and main results}

Let $G$ be a split connected reductive group over any field $k$. Let $W$ be the Iwahori--Weyl
group of $LG(k) = G(k(\!(t)\!))$, and for each $r$-tuple $w_\bullet = (w_1, \dots, w_r) \in W^r$ and choice of standard parahoric subgroup $\calP \subset LG(k)$, consider the convolution morphism 
$$
m_{w_\bullet, \calP}\colon X_\calP(w_\bullet) := X_\calP(w_1) \widetilde{\times} \cdots \widetilde{\times} X_\calP(w_r) \longrightarrow X_\calP(w_*)    
$$
defined on the twisted product of Schubert varieties $X_\calP(w_i) \subset \Flag_\calP$ (see Sections~\ref{Notation_sec} and~\ref{Review_sec}).
Such morphisms have long played an important role in the geometric Langlands program and in the study of the geometry of Schubert varieties.  For example, if 
$w_\bullet = (s_1,\dots, s_r)$ is a sequence of simple affine reflections, $w = s_1\cdots s_r$ is a reduced word, and $\calP$ is the standard Iwahori subgroup $\calB$, then $X_\calB(s_\bullet) \rightarrow X_\calB(w)$ is the Demazure resolution (of singularities) of $X_\calB(w)$.   If $\calP = L^+G$ is the positive loop group and $w_\bullet = \mu_\bullet = (\mu_1, \dots, \mu_r)$ is a tuple of cocharacters in $G$, the corresponding convolution morphism is used to define the convolution of $L^+G$-equivariant perverse sheaves on the affine Grassmannian $\Gr_G = LG/L^+G$, and hence it plays a key role in the geometric Satake correspondence.

Numerous applications stem from the study of the fibers of convolution morphisms, their dimensions and irreducible components, and possible pavings of them by affine spaces or related spaces.  This article will focus on pavings of fibers by affine spaces, or by closely related spaces. We recall that a variety $X$ is {\em paved by varieties in a class $\mathcal C$} provided that there exists a finite exhaustion by closed subvarieties $\emptyset = X_0 \subset X_1 \subset \cdots \subset X_l = X$ such that each locally closed difference $X_i - X_{i-1}$ for $1 \leq i \leq r$ is isomorphic to a member of the class $\mathcal C$.

The fact that the fibers of Demazure resolutions admit pavings by affine spaces was for a long time a folklore result, until a proof appeared in \cite{H05} and later more generally in \cite{dHL18} (see \cite[Theorem~2.5.2]{dHL18}).
This has been used in proving various parity vanishing and purity results in Kazhdan--Lusztig theory, see \cite{H05, dHL18}, and in the geometric Satake correspondence; see \cite{Ga01, MV07, Ri14}. The paving of certain fibers related to the affine Grassmannian for $\GL_n$ gives a different approach to paving by affines of some Springer--Spaltenstein varieties, which are certain partial Springer resolutions of the nilpotent cone for $\GL_n$; see \cite[Proposition~8.2 and what follows]{Hai06}  
and \cite{Sp}.

One could conjecture that all fibers of general convolution morphisms $X_\calP(w_\bullet) \rightarrow X_\calP(w_*)$ are paved by affine spaces.  In the special case of a sequence of minuscule cocharacters $w_\bullet = \mu_\bullet$ and the associated convolution morphism  $X_{L^+G}(\mu_\bullet) \to X_{L^+G}(|\mu_\bullet|)$ for the affine Grassmannian $\Gr_G = LG/L^+G$, this was proved in \cite[Corollary~1.2]{Hai06}.  In general for the affine Grassmannian, it is not known which fibers are paved by affine spaces (see \cite[Question 3.9]{Hai06}). The existence of an affine space paving of fibers of $m_{w_\bullet, \calP}$ in the general case seems to be an interesting open question---and the author is not aware of any counterexamples.  One can consider the analogous question of when fibers of {\em uncompactified} convolution morphisms $Y_\calP(w_\bullet) \rightarrow X_\calP(w_*)$ are paved by affine spaces.  This turns out to usually fail (for examples, see Remarks~\ref{s,s_eg} and~\ref{KLM_eg} below). However, a weaker result does always hold. (As was implicit in the above discussion, in the following statements, all fibers are the scheme-theoretic fibers endowed with their induced reduced subscheme structure.)

\begin{thm} \label{thmA}
Every fiber of a convolution morphism $Y_\calP(w_\bullet) \rightarrow X_\calP(w_*)$ is paved by finite products of copies of $\mathbb A^1$ and $\mathbb A^1 - \mathbb A^0$. 
\end{thm}

As a corollary, we obtain the following result on fibers of the usual convolution morphisms.

\begin{cor} \label{corB}
Every fiber of any convolution morphism $X_\calP(w_\bullet) \rightarrow X_\calP(w_*)$ is paved by finite products of copies of  $\mathbb A^1$ and $\mathbb A^1 - \mathbb A^0$.
\end{cor}

The previous two results show that the fibers in question are similar to  {\em cellular $k$-schemes}, in the sense of \cite[Definition~3.1.5]{RS20}. We adopt a similar terminology and declare that they admit {\em cellular pavings}. (This reflects a technical point: Here, we do not prove the {\em stratification property}, the property that the closure of each stratum is a union of strata.) A weaker version of Corollary~\ref{corB} was stated without proof in \cite[Remark~2.5.4]{dHL18}. The proof is given here in Sections~\ref{P=B_sec},~\ref{thmA_sec}, and~\ref{corB_sec}. 

One situation where paving by affine spaces is known is given by the following result.

\begin{thm}\label{thmC} Suppose $w_\bullet = s_\bullet = (s_1, s_2, \dots, s_r)$ is a sequence of simple reflections with Demazure product $s_* = s_1 * s_2 * \cdots * s_r$.  Then the fibers of $X_\calB(s_\bullet) \rightarrow X_{\calB}(s_*)$ are paved by affine spaces.
\end{thm}

This theorem was proved in \cite[Theorem~2.5.2]{dHL18}.  However, here we give a different proof, which has the advantage that it can be easily adapted to prove the special case of Theorem~\ref{thmA} where $\calP = \calB$ and every $w_i$ is a simple reflection. This in turn is used to prove the general case of Theorem~\ref{thmA}.

The results above should all have analogues at least for connected reductive groups $G$ which are defined and tamely ramified over a field $k(\!(t)\!)$ with $k$ perfect (see Remark~\ref{k_perf_rem}).  The proofs will necessarily be more involved and technical, and the author expects them to appear in a separate work.

\smallskip

In Section~\ref{Z_sec}, we extend all the preceding results over $\bbZ$.  We prove in that section the following result (Theorem~\ref{CvdHS_gen}), the second part of which recovers \cite[Theorem~1.2]{CvdHS+}.

\begin{thm} \label{thmD} Assume $G_\bbZ \supset B_\bbZ \supset T_\bbZ$ is a connected reductive group over $\bbZ$ with Borel pair defined over $\bbZ$. Consider a parahoric subgroup $\calP_\bbZ$ and the associated Schubert schemes $X_{\calP, \bbZ}(w) \subset \Flag_{\calP,\bbZ}$. The convolution morphisms attached to $w_\bullet = (w_1, \dots, w_r) \in W^r$ may be constructed over $\bbZ$,  
$$
m_{w_\bullet, \calP_\bbZ}\colon X_{\calP, \bbZ}(w_\bullet) \longrightarrow X_{\calP, \bbZ}(w_*), 
$$
and for any $v \leq w_*$, the reduced fiber $m_{w_\bullet, \calP_\bbZ}^{-1}(v\,e_{\calP_\calZ})$ has a cellular paving over $\bbZ$.  Furthermore, for any standard parabolic subgroup $P_\bbZ = M_\bbZ N_\bbZ$ and any pair $(\mu, \lambda) \in X_*(T)^+ \times X_*(T)^{+_M}$, the subscheme $L^+M_\bbZ LN_\bbZ x_\lambda \cap L^+G_\bbZ x_\mu$ of the affine Grassmannian $\Gr_{G, \bbZ}$, with its reduced subscheme structure, has a cellular paving over $\bbZ$.
\end{thm}

\smallskip

\subsection*{Leitfaden}  Here is an outline of the contents of this article. In Sections~\ref{Notation_sec} and~\ref{Review_sec}, we give our notation and recall the basic definitions related to convolution morphisms.  The main idea of the proof of Theorem~\ref{thmA} is to prove it by induction on $r$: One projects from the fiber onto the $r-1$ term in the twisted product; then one needs to show that the image is paved by locally closed subvarieties,  each of which has a $\mathcal C$-paving, and over which the aforementioned projection morphism is trivial.  The strategy of proof is given in more detail in Section~\ref{thmC_proof}. The required triviality statements are proved in Sections~\ref{Fact_sec} and~\ref{Strat_triv_sec}.  The core of the article is found in Sections~\ref{P=B_sec}--~\ref{corB_sec}.  First, Theorem~\ref{thmC} is proved in Section~\ref{thmC_proof}, and this proof is then adapted to prove the special case of Theorem~\ref{thmA} for $\calP = \calB$ and all $w_i$ simple reflections, in Section~\ref{ThmA_spec_case_subsec}. This is used to deduce the special case of Theorem~\ref{thmA} with $\calP = \calB$ in Section~\ref{ThmA_P=B_proof}. Finally, the general case of Theorem~\ref{thmA} is proved in Section~\ref{ThmA_general}, using the previous special cases as stepping stones.  In Section~\ref{corB_sec}, we quickly deduce Corollary~\ref{corB} from Theorem~\ref{thmA}. In Section~\ref{HA_app}, we give an application to structure constants for parahoric Hecke algebras.  In Section~\ref{Z_sec}, we develop all the needed machinery to extend the above results over $\bbZ$. The paper ends with Errata for \cite{dHL18} in Section~\ref{Errata}.

\subsection*{Acknowledgments} I express my thanks to Thibaud van den Hove, whose questions about \cite[Remark~2.5.4]{dHL18} prompted me to write up these results. I  also thank him for helpful comments on an early version of this paper, and for giving me access to an advance copy of the revised version of \cite{CvdHS+}. I am grateful to the referee for helpful remarks and suggestions.

\section{Notation} \label{Notation_sec}

Generally speaking, we follow the same notation and conventions as \cite{dHL18}. Let $G$ be a split connected reductive group over a field $k$ with algebraic closure $\bar{k}$ and separable closure $k^{\sep}$.  Fix a Borel pair $G \supset B \supset T$, also split and defined over $k$. This gives rise to the based absolute root system 
$
(X^*(T) \supset \Phi, X_*(T) \supset \Phi^\vee, \Delta),
$
the real vector space $V = X_*(T) \otimes \mathbb R$, and the canonical perfect pairing $\langle \cdot, \cdot \rangle\colon X^*(T) \times X_*(T) \rightarrow \mathbb Z$.
The affine roots $\Phi_{\aff} = \{ a = \alpha + n \, | \, \alpha \in \Phi, ~~ n \in \mathbb Z\}$ are affine-linear functionals on $V$.  We denote by ${\bf 0}$ the facet containing the origin in $V$ and the $B$-dominant Weyl chamber $\mathfrak C = \{ v \in V \, | \, \langle \alpha, v \rangle > 0, \,\forall \alpha \in \Delta\}$. We denote the set of dominant cocharacters by $X_*(T)^+ := X_*(T) \cap \overline{\mathfrak C}$, where $\overline{\mathfrak C}$ is the closure of $\mathfrak C$ in $V$. We also fix the base alcove ${\bf a} \subset \mathfrak C$ whose closure contains ${\bf 0}$.  The positive simple affine roots $\Delta_{\aff}$ are the minimal affine roots $a = \alpha + n$ taking positive values on ${\bf a}$.  We use the convention that $\lambda \in X_*(T)$ acts on $V$ by translation by $\lambda$.  The finite Weyl group is the Coxeter group $(W_0, S)$ generated by the simple reflections
$\{ s_\alpha \in S\}$ on $V$, for $\alpha \in \Delta$; the group $W_0$ fixes ${\bf 0}$. The extended affine Weyl group $W = X_*(T) \rtimes W_0$ acts on $V$ and hence on the set $\Phi_{\aff}$ by precomposition.  Let $(W_{\aff}, S_{\aff})$ denote the Coxeter group generated by $S_{\aff}$, the simple affine reflections $s_a$ for $a \in \Delta_{\aff}$.  It has a Bruhat order $\leq$ and a length function $\ell\colon W_{\aff} \rightarrow \mathbb Z_{\geq 0}$. Let $\Omega \subset W$ be the subgroup stabilizing ${\bf a} \subset V$. The group decomposition $W = W_{\aff} \rtimes \Omega$ allows us to extend $\leq$ and $\ell$ from $W_{\aff}$ to $W$, by declaring $\Omega$ to be the set of length zero elements in $W$. 

Fix the field $F = k(\!(t)\!)$ and ring of integers $\mathcal O = k[\![t]\!]$. The Iwahori--Weyl group $N_G(T)(F)/T(\mathcal O)$ may be naturally identified with $W$. We choose once and for all lifts of $w \in W_0$ in $N_GT(\mathcal O)$, and we lift $\lambda \in X_*(T)$ to the element $t^\lambda := \lambda(t) \in T(F)$.  Altering these lifts by any elements in $T(\mathcal O)$ does not affect anything in what follows.
The group $N_G(T)(F)$ acts naturally on the apartment $X_*(T) \otimes \mathbb R$, and our convention  is that $t^\lambda$ acts by translation by $-\lambda$. This is the same as the convention for identifying the Iwahori--Weyl group with the extended Weyl group which is used in \cite{BT72}.

We define the loop group $LG$ (resp., positive loop group $L^+G$) to be the group ind-scheme (resp., group scheme) over $k$ representing the group functor on $k$-algebras $LG(R) = G(R(\!(t)\!))$ (resp., $L^+G(R) = G(k[\![t]\!])$). For a facet ${\bf f}$ contained in the closure of ${\bf a}$, we obtain the ``standard'' parahoric group scheme $P_{\bf f}$ (see \cite{BT84, HR08}).  We often write $\calP := L^+P_{\bf f}$ and regard this as a (standard) parahoric group in $LG$.  Note that $L^+G = L^+P_{\bf 0}$. The (standard) Iwahori subgroup will be denoted by $\calB := L^+P_{\bf a}$.  Let $W_{\bf f} =: W_{\calP} \subset W_{\aff}$ be the subgroup which fixes ${\bf f}$ pointwise; it is a Coxeter group generated by the simple affine reflections which fix ${\bf f}$. The Bruhat order $\leq$ on $W$ descends to a Bruhat order $\leq$ on coset spaces such as $W_\calP \backslash W / W_\calP$ and $W/W_\calP$. Let $^{\bf f} W ^{\bf f}$ denote the elements $w \in W$ which are the unique $\leq$-maximal elements in their double cosets $W_\calP w W_\calP$.  

The partial affine flag variety is by definition the \'{e}tale sheafification of the presheaf on the category ${\aff}_k$ of affine schemes $\Spec(R)$  over $k$ given by $R \mapsto LG(R)/L^+P_{\bf f}(R)$. It is represented by an ind-projective ind-scheme denoted simply by $\Flag_\calP = LG/L^+P_{\bf f}$, and it carries a left action by $\calP = L^+P_{\bf f}$. Denote by $e_\calP$ its natural base point.  It is well known (see \textit{e.g.} \cite{HR08})
that for any two standard parahoric subgroups ${\mathcal Q}$ and~$\calP$, we have a natural bijection on the level of $k$-points and $k^{\sep}$-points 
\begin{equation} \label{BT_decomp}
{\mathcal Q}(k) \backslash \Flag_\calP(k) = W_{\mathcal Q} \backslash W/W_\calP = {\mathcal Q}(k^{\sep}) \backslash \Flag_\calP(k^{\sep}).
\end{equation}
The elements of Bruhat--Tits theory used in \cite[Proposition~8, Remark~9]{HR08} work for split groups without any assumption that the residue field $k$ is perfect (\textit{cf.} also Remark~\ref{k_perf_rem}).  Alternatively, for split $G$, (\ref{BT_decomp}) can be proved directly for any residue field $k$ (including $\bar{k}$), using BN-pair relations.

For $w \in W$, let $Y_\calP(w)$ (resp., $Y_{\calB \calP}(w)$) denote the $\calP$-orbit (resp., $\calB$-orbit) of $we_\calP$ in $\Flag_\calP$.  When $\calP = \calB$, we will often omit the subscripts. Define the {\em Schubert variety} $X_\calP(w)$ to be the Zariski closure of $Y_\calP(w) \subset \Flag_\calP$, endowed with reduced structure. Similarly, define $X(w) = X_\calB(w)$ and $X_{\calB \calP}(w)$.

In the part of this paper where we work over a field $k$, the schemes which arise are finite-type separated schemes over $k$ (not necessarily irreducible). We will always give them a reduced structure, and we will call them ``varieties.''  The morphisms of varieties we consider will always be defined over $k$, and the varieties and the morphisms between them will usually be described on the level of points in an unspecified algebraic closure of $k$.

In this article, we use the following notation: An equality  $\bigsqcup_i Z_i = Z$ denotes a finite (or possibly countable) paving of a scheme or ind-scheme $Z$ into disjoint locally closed subschemes $Z_i$. Note that the underlying sets are disjoint but there could be closure relations between them---they are not necessarily coproducts in the category of topological spaces.

\section{Review of convolution morphisms} \label{Review_sec}

For $w \in W$, define $\overline{\calP w \calP} = \bigsqcup_{v \leq w} \calP v \calP$, where $v$ ranges over elements $v \in W_\calP \backslash W /W_\calP$. For any $r$-tuple $w_\bullet = (w_1, \dots, w_r) \in W^r$, we define $X_\calP(w_\bullet)$ to be the quotient of $\calP^r = (L^+P_{\bf f})^r$ acting on
$$
\overline{\calP w_1 \calP} \times \overline{\calP w_2 \calP} \times ~ \cdots ~ \times \overline{\calP w_r \calP}
$$
by the right action
\begin{equation} \label{r-fold_action}
(g_1, g_2, \dots, g_r) \cdot (p_1, p_2, \dots, p_r)  := \left(g_1 p_1, p^{-1}_1g_2 p_2, \dots, p^{-1}_{r-1} g_r p_r\right).
\end{equation}
We define $Y_\calP(w_\bullet)$ similarly, with each $\overline{\calP w_i \calP}$ replaced by $\calP w_i \calP$. The quotients should be understood as \'{e}tale sheafifications of presheaf quotients on the category ${\aff}_k$. It is well known that $X_\calP(w_\bullet)$ (resp., $Y_\calP(w_\bullet)$) is represented by an irreducible projective (resp., quasi-projective) $k$-variety.

We regard the above objects as ``twisted products'': $X_\calP(w_\bullet) = X_\calP(w_1) \widetilde{\times} X_\calP(w_2) \widetilde{\times} \cdots \widetilde{\times} X_\calP(w_r)$ (resp., $Y_\calP(w_\bullet) = Y_\calP(w_1) \widetilde{\times} Y_\calP(w_2) \widetilde{\times} \cdots \widetilde{\times} Y_\calP(w_r)$), consisting of tuples $(g_1\calP, g_2\calP, \dots, g_r\calP)$ such that $g_{i-1}^{-1}g_i \in \overline{\calP w_i \calP}$ (resp., \,$\calP w_i \calP$) for all $1 \leq i \leq r$ (here $g_0 = 1$ by convention).

Recall that the Demazure product $W^r \rightarrow W$, $(w_1, w_2, \dots, w_r) \mapsto w_1 * w_2 * \cdots * w_r$ is an associative operation.  It induces an associative $r$-fold product $(\,^{\bf f}W^{\bf f})^r \rightarrow \, ^{\bf f}W^{\bf f}$; see \textit{e.g.} \cite[Section~4]{dHL18}.  Given any $w \in W$, let $^{\bf f}w^{\bf f}$ denote the unique $\leq$-maximal element in $W_\calP w W_\calP$.  Given $w_\bullet = (w_1, w_2, \dots, w_r) \in W^r$, define
$$
w_* := \,^{\bf f}w_1^{\bf f} * \,^{\bf f}w_2^{\bf f} * \cdots * \,^{\bf f}w_r^{\bf f}.
$$
Then the multiplication map $(LG)^r \rightarrow LG$ descends to the quotient and defines the {\em convolution morphisms}
\begin{align*}
m_{w_\bullet} &= m_{w_\bullet, \calP} \colon X_\calP(w_\bullet) \longrightarrow X_\calP(w_*),  \\
p_{w_\bullet} &= p_{w_\bullet, \calP} \colon Y_\calP(w_\bullet) \longrightarrow X_\calP(w_*);  
\end{align*}
see \cite[Section~4]{dHL18}. We might describe those of the  second kind as {\em uncompactified convolution morphisms}.

\section{Consequences of a factorization of the pro-unipotent Iwahori subgroup} \label{Fact_sec}

Recall that $\bf f$ is a facet in the closure of ${\bf a}$ and the facets ${\bf f}$ and ${\bf a}$ give rise to the Iwahori and parahoric subgroups $\calB = L^+P_{\bf a}$ and $\calP = L^+P_{\bf f}$. The group $\mathcal U$ is the pro-unipotent radical of the Iwahori subgroup $\calB$; it is the preimage of the unipotent radical $U \subset B$ under the natural homomorphism $\calB \rightarrow B$ induced by $t \mapsto 0$; see \cite[Section~3.7]{dHL18}. Also recall  that $\overline{\mathcal U}_{\calP} = L^{--}P_{\bf f}$ is the ind-affine group ind-scheme defined in \cite[Definition~3.6.1]{dHL18}, called the negative parahoric loop group. The definition is given over $\mathbb Z$ in (\ref{neg_parahoric_def}).

For an affine root $a$, the notation $a \overset{\bf f}{>} 0$ means that $a$ takes positive values on the facet ${\bf f}$.  When ${\bf f} = {\bf a}$, we usually simply write $a > 0$.  The notation $a \overset{\bf f}{\geq} 0$, $a \overset{\bf f}{<} 0$, \textit{etc.}, has the obvious meaning.

\begin{prop} \label{dHL.3.7.4} Let $\calP \supset \calB$ be any fixed parahoric subgroup as above, and let $v \in W$ be an arbitrary element. Let ${\bf f}$ be the facet in the closure of the base alcove ${\bf a}$ which corresponds to $\calP$.
\begin{a-enumerate}
\item\label{dHL.3.7.4a} We have a \textup{(}non-commuting\textup{)} factorization of group functors
$$
\mathcal U = \left(\mathcal U \cap \,^v\overline{\mathcal U}_\calP\right) \, \cdot \, (\mathcal U \cap \,^v \calP).
$$
\item\label{dHL.3.7.4b} There is an isomorphism of schemes $\mathcal U \cap \,^v\overline{\mathcal U}_\calP \cong \prod_a U_a$, where $U_a$ ranges over the affine root groups corresponding to affine roots with $a > 0$ and $v^{-1}a \overset{\bf f}{<} 0$, and the product is taken in any order.
\end{a-enumerate}
\end{prop}

\begin{proof} This is \cite[Proposition~3.7.4]{dHL18}. The proof over $\mathbb Z$ given in Proposition~\ref{dHL.3.7.4_Z} works here as well.
\end{proof}

\begin{rmk} \label{k_perf_rem}
Usually the hypothesis that $k$ is perfect is implicit in Bruhat--Tits theory and the theory of parahoric subgroups: All residue fields of the complete discretely valued fields $F$ one works over should be assumed to be perfect so that Steinberg's theorem applies to show that every reductive group over the completion $\breve{F}$ of a maximal unramified extension of $F$ is quasi-split. (This assumption on residue fields is missing from \cite{HR08} and should be added.  I am grateful to Gopal Prasad for pointing out this oversight.)  Since we are assuming our group $G$ is already split over $k$, it is automatically quasi-split over $\breve{k(\!(t)\!)} = k^{\sep}(\!(t)\!)$. Therefore, we do not need to assume $k$ is perfect when invoking Bruhat--Tits theory for $G$.  Note that the hypothesis that $k$ is perfect appears to be used in the proof of \cite[Proposition~3.7.4]{dHL18} since that proof relies on \cite[Remark~3.1.1]{dHL18}.  However, the latter actually holds for all $k$: We see the key point that $B$ is {\em $k$-triangularizable} in the sense of \cite[Section~14.1]{Spr} by invoking \cite[Propositions~16.1.1 and~14.1.2]{Spr} applied to $B$.
\end{rmk}

\begin{lem} \label{root_prod} Assume that $v$ is right-$\bf f$-minimal, i.e., it is the unique minimal element in its coset $v W_{\langle {\bf f}\rangle}$, where $W_{\langle \bf f \rangle}$ is the Coxeter subgroup of\, $W_{\aff}$ which fixes $\bf f$ pointwise. Then for any positive affine root $a >0$, we have
$$
v^{-1}a \overset{\bf f}{<} 0 \, \Longleftrightarrow \, v^{-1}a < 0.
$$
\end{lem}

\begin{proof}
The implication $(\Rightarrow)$ uses only that ${\bf f}$ belongs to the closure of ${\bf a}$ and holds for any $v \in W$.

Next we prove $(\Leftarrow)$.   Assuming $v^{-1}a < 0$, we wish to prove $v^{-1}a \overset{\bf f}{<} 0$.  Suppose on the contrary that $v^{-1}a \overset{\bf f}{\geq} 0$. Combining it with $v^{-1}a  < 0$, we deduce $v^{-1}a \overset{\bf f}{=} 0$, that is, $v^{-1}s_a v \in W_{\langle \bf f \rangle}$.   Since $v$ is right-$\bf f$-minimal and $s_a v \in vW_{\langle \bf f \rangle}$, we deduce that $s_a v > v$.  On the other hand, since $a$ is positive on ${\bf a}$ and $a$ is negative on $v\bf a$, we see that $v{\bf a}$ and ${\bf a}$ are on opposite sides of the affine root hyperplane $H_a$, which means $s_a v < v$, so we have a contradiction.
\end{proof}

\begin{prop} \label{BP_comp}
If $v \in W$ is right-$\bf f$-minimal, then we have isomorphisms
$$
Y_{\calB \calP}(v)~\cong~\mathcal U \cap \,^v\overline{\mathcal U}_\calP~\cong~\mathcal U \cap \,^v\overline{\mathcal U}~\cong~Y_{\calB \calB}(v).
$$
\end{prop}

\begin{proof}
Since $v$ normalizes $T(\mathcal O)$, we have $Y_{\calB \calP}(v) = \mathcal U v \calP/\calP$, which identifies with $\mathcal U \cap \,^v\overline{\mathcal U}_\calP$ by Proposition~\ref{dHL.3.7.4}\eqref{dHL.3.7.4a} and by the fact that $\overline{\mathcal U}_{\calP} \rightarrow \Flag_\calP$ is an (open) immersion; see \textit{e.g.} \cite[Theorem~2.3.1]{dHL18}. This is identified with $\mathcal U \cap \,^v\overline{\mathcal U}$ by Proposition~\ref{dHL.3.7.4}\eqref{dHL.3.7.4b} and Lemma~\ref{root_prod}.
\end{proof}

\begin{rmk}
The proof of Proposition~\ref{BP_comp} given here ultimately relies on the negative parahoric loop group introduced in \cite{dHL18}.  Another proof which is more general and which avoids this reliance is given in \cite[Lemma~3.3]{HaRi2}. We give the proof above because it is an almost immediate consequence of Proposition~\ref{dHL.3.7.4}, which we need anyway to establish Proposition~\ref{strat_triv} below.  
\end{rmk}

\section{Stratified triviality of convolution morphisms} \label{Strat_triv_sec}

\begin{dfn}
Let $f\colon X \rightarrow Y$ be a morphism of schemes over a base scheme $S$. For a locally closed $S$-subscheme $Z \subset Y$ and an $S$-point $z\colon S \rightarrow Z$, write $f^{-1}(z) = f^{-1}(Z) \times_{f,Z, z} S$. We say that $f$ is {\em trivial} over $Z$ if for some point $z \in Z(S)$ \textup{(}equivalently, for all points $z \in Z(S)$\textup{)}, there is an isomorphism of schemes $\phi_z\colon f^{-1}(Z) \overset{\lowsim}{\rightarrow} f^{-1}(z) \times_S Z$ such that the diagram 
\begin{equation} \label{triangle}
\xymatrix{
f^{-1}(Z) \ar[rr]^{\phi_z}_{\sim} \ar[dr]_f && f^{-1}(z) \times_S Z \ar[dl]^{{\rm pr}_2} \\
 & Z &
}
\end{equation}
commutes and such that, writing $\phi_z = (\psi_z, f)$, the morphism $\psi_z$ is a retraction (a left-inverse) of the canonical morphism $i_z\colon f^{-1}(z) \rightarrow f^{-1}(Z)$.
\end{dfn}

To justify the definition, we need the following lemma.

\begin{lem} \label{triviality_lem}
For every pair $z, z' \in Z(S)$, the data $(\phi_z, \psi_z)$ for $z$ give rise to corresponding data $(\phi_{z'}, \psi_{z'})$ for $z'$. This shows that if $f$ is trivial with respect to one choice of\, $z \in Z(S)$, it is trivial with respect to any other choice $z' \in Z(S)$.
\end{lem}

\begin{proof}
Suppose $\phi_z = (\psi_z, f)$ is given, and suppose $z'\colon S \rightarrow Z$ is any $S$-point of $Z$.  Then $\phi_z^{-1}$ as in (\ref{triangle}) induces an $S$-isomorphism $\psi_{z',z}$ from the pull-back of ${\rm pr}_2$ along $z'$ to the pull-back of $f$ along $z'$, that is,
$$
\psi_{z', z}\colon f^{-1}(z) \overset{\lowsim}{\longrightarrow} f^{-1}(z'). 
$$
Then $\psi_{z'} := \psi_{z', z} \circ \psi_z$ gives a $Z$-isomorphism $\phi_{z'} := (\psi_{z'}, f)$ attached to $z'$, and the assumption that $\psi_z$ is a retraction of $i_z$ implies that $\psi_{z'}$ is likewise a retraction of the canonical morphism $i_{z'}\colon f^{-1}(z') \rightarrow f^{-1}(Z)$.
\end{proof}

When we are working in the category of $k$-varieties (as in the next proposition), we use the same terminology, but it is understood that objects and morphisms are in that category, and it may or may not be possible to upgrade the statements to the category of $k$-schemes, that is, without taking reduced structures.

\begin{prop} \label{strat_triv}
The morphism $m_{w_\bullet}\colon X_\calP(w_\bullet) \rightarrow X_\calP(w_*)$ is trivial over every $\calB$-orbit in its image.
\end{prop}

\begin{proof}
Writing $m := m_{w_\bullet}$, we prove the triviality of the map $m$ over $\mathcal B$-orbits contained in its image. Assume $Y_{\calB \calP}(v) \subset X_\calP(w_*)$.  By Proposition~\ref{BP_comp}, an element $\mathcal P' \in Y_{\calB \calP}(v)$ can be written in the form 
$$
\mathcal P' = {uv}\calP
$$
for a unique element $u \in \mathcal U \cap \, ^v\overline{\mathcal U}_\calP$. 

We can then define an isomorphism
\begin{equation} \label{5.3_eqn}
m^{-1}(Y_{\calB \calP}(v)) ~ \overset{\lowsim}{\longrightarrow} ~ \, m^{-1}(v\mathcal P) \times Y_{\calB \calP}(v)
\end{equation} 
by sending $(\mathcal P_1, \dots, \mathcal P_{r-1}, {uv}\mathcal P)$ to $(\, {u^{-1}}\mathcal P_1, \cdots, {u^{-1}}\mathcal P_{r-1}, {v}\mathcal P) \times {uv}\mathcal P$.  Obviously, the first factor belongs to $m^{-1}(v\mathcal P)$.  

This is a situation where the statement and proof work in the category of $k$-schemes, not just the category of $k$-varieties -- this uses the schematic decomposition proved in Proposition~\ref{dHL.3.7.4}\eqref{dHL.3.7.4a}. But here we are giving all fibers their reduced structure; this makes sense because the scheme-theoretic fiber product over $k$ on the right-hand side of (\ref{5.3_eqn}) is automatically reduced, even when $k$ is not perfect, because $Y_{\calB \calP}(v)$ is schematically an affine space, thanks to Proposition~\ref{dHL.3.7.4}\eqref{dHL.3.7.4b}.
\end{proof}

\section{Paving results for the case \texorpdfstring{$\boldsymbol{\calP = \calB}$}{P=B}} \label{P=B_sec}

\subsection{BN-pair relations and lemmas on retractions}\label{bnpair}

The following statements can be interpreted at the level of $k$- or $\bar{k}$-points, but we will suppress this from the notation. Recall that given $\mathcal B_1 =  {g_1}\mathcal B$,  $\mathcal B_2 = {g_2}\mathcal B$, and $w \in W$, we say the pair $(\mathcal B_1, \mathcal B_2)$ is in relative position $w$ (and we write $\mathcal B_1 \, \overset{w}{\longdash} \,\mathcal B_2$) if and only if $g_1^{-1}g_2 \in \mathcal B w \mathcal B$.
We write
$$
\mathcal B_1 \, \overset{\leq w}{\longdash} \,\mathcal B_2 \,\,\,\,\,\,\,\, \mbox{if and only if} \,\,\,\,\,\,\,\, \mathcal B_1 \, \overset{v}{\longdash} \, \mathcal B_2 \,\,\,\,\mbox{for some $v \leq w$}.
$$
We have
\begin{equation*}
Y_\mathcal B(w) = \left\{ \mathcal B' ~ | ~ \mathcal B \, \overset{w}{\longdash}\, \mathcal B' \right\} \hspace{.25in} \mbox{and} \hspace{.25in} X_\mathcal B(w) = \left\{ \mathcal B' ~ | ~ \mathcal B \, \overset{\leq w}{\longdash} \, \mathcal B' \right\}.
\end{equation*}

The BN-pair relations hold for $v \in W$ and $s \in S_{\aff}$:
\begin{align} \label{BN-pair_eq}
\mathcal B v \mathcal B s \mathcal B &= \begin{cases} \mathcal B vs \mathcal B & \text{if } v < vs, \\
  \mathcal B vs \mathcal B \cup \mathcal B v \mathcal B & \text{if } vs < v.
\end{cases} \\
\notag
\end{align}
Note that for every $v \in W$ and $s \in S_{\aff}$, there is an isomorphism $\{ \mathcal B' ~ | ~ v\mathcal B \,\, \overset{\leq s}{\longdash} \,\mathcal B' \} \cong \mathbb P^1$ and $\{ \mathcal B' ~ | ~ v\mathcal B \, \,\overset{\leq s}{\longdash} \,\mathcal B' \} \subset Y_\mathcal B(v) \cup Y_\mathcal B(vs)$.

\begin{lem}\label{BN_pair_lem}
Suppose $s \in S_{\aff}$ and $v \in W$.
\begin{enumerate}
\item If\, $v < vs$, then $\{ \mathcal B' ~ | ~ v\mathcal B \,\, \overset{\leq s}{\longdash} \, \mathcal B' \} \cap Y_\mathcal B(v) = \{ v\mathcal B \} \cong \mathbb A^0$.
\item If\, $v < vs$, then $\{ \mathcal B' ~ | ~ v \mathcal B \,\, \overset{\leq s}{\longdash} \, \mathcal B' \} \cap Y_\mathcal B(vs) \cong \mathbb A^1$.
\item If\, $vs < v$, then $\{ \mathcal B' ~ | ~ v\mathcal B \,\, \overset{\leq s}{\longdash} \, \mathcal B' \} \cap Y_\mathcal B(v) \cong \mathbb A^1$.
\item If\, $vs < v$, Then $\{ \mathcal B' ~ | ~ v\mathcal B \,\, \overset{\leq s}{\longdash} \, \mathcal B' \} \cap Y_\mathcal B(vs) = \{ {vs}\mathcal B\} \cong \mathbb A^0$.
\end{enumerate}
\end{lem}

\begin{proof} This is obvious from properties of the retraction map from the building associated to $G$ onto the apartment corresponding to $T$, with respect to an alcove in that apartment. A reference for how such retractions ``work'' is \cite[Section~6]{HKM}.
\end{proof}

In a similar way, we get an analogous lemma.

\begin{lem}\label{strict-BN_pair_lem}
Suppose $s \in S_{\aff}$ and $v \in W$.
\begin{enumerate}
\item\label{sBNpl-1} If\, $v < vs$, then $\{ \mathcal B' ~ | ~ v\mathcal B \,\, \overset{s}{\longdash} \, \mathcal B' \} \cap Y_\mathcal B(v) = \emptyset$.
\item\label{sBNpl-2} If\, $v < vs$, then $\{ \mathcal B' ~ | ~ v \mathcal B \,\, \overset{s}{\longdash} \, \mathcal B' \} \cap Y_\mathcal B(vs) \cong \mathbb A^1$.
\item\label{sBNpl-3} If\, $vs < v$, then $\{ \mathcal B' ~ | ~ v\mathcal B \,\, \overset{s}{\longdash} \, \mathcal B' \} \cap Y_\mathcal B(v) \cong \mathbb A^1- \mathbb A^0$.
\item\label{sBNpl-4} If\, $vs < v$, then $\{ \mathcal B' ~ | ~ v\mathcal B \,\, \overset{s}{\longdash} \, \mathcal B' \} \cap Y_\mathcal B(vs) = \{ {vs}\mathcal B\} \cong \mathbb A^0$.
\end{enumerate}
\end{lem}

\subsection{Proof of Theorem~\ref{thmC}} \label{thmC_proof}

Since $\mathcal B$ is understood, we will write $Y(w_\bullet)$ for $Y_\mathcal B(w_\bullet)$ and $X(w_\bullet)$ for $X_\mathcal B(w_\bullet)$ in what follows. Let $s_\bullet = (s_1, \dots, s_r) \in \mathcal S^r$. There is no requirement here that $s_1 \cdots s_r$ be reduced. Recall the subvariety $X(s_\bullet) \subset (\mathcal G/\mathcal B)^r$ which consists of the $r$-tuples $(\mathcal B_1, \dots, \mathcal B_r)$ such that $\mathcal B_{i-1} \overset{\leq s_i}{\longdash} \mathcal B_i$ for all $i = 1, \dots, r$ (with the convention that $\mathcal B_0 = \mathcal B$). 

We are going to prove the paving by affine spaces of the fibers of the morphism
$$
m\colon X(s_\bullet) \longrightarrow X(s_*) \subset \mathcal G/\mathcal B, \quad 
(\mathcal B_1, \dots, \mathcal B_r) \longmapsto \mathcal B_r.
$$
We proceed by induction on $r$. The case $r = 1$ is trivial, 
so we assume $r > 1$ and that the theorem holds for $r-1$. Let $s'_* := s_1 * \cdots * s_{r-1}$. Let $s'_\bullet = (s_1, \dots, s_{r-1})$. By our induction hypothesis, the theorem holds for 
$$
m'\colon X(s'_\bullet) \longrightarrow  X(s'_*), \quad 
(\mathcal B_1, \dots, \mathcal B_{r-1}) \longmapsto \mathcal B_{r-1}.
$$
Now suppose $v \leq s_*$, so that $v\mathcal B \in \Im(m)$. For an element $(\mathcal B_1, \dots, \mathcal B_{r-1}, v\mathcal B) \in m^{-1}(v\mathcal B)$, we have 
$$
\mathcal B \, \overset{v}{\longdash} \, v\mathcal B \, \,\overset{\leq s_r}{\longdash} \, \mathcal B_{r-1}.$$
It follows from the BN-pair relations that $\mathcal B_{r-1} \in Y(v) \cup Y(vs_r)$. We consider the map
\begin{align*}
\xi\colon m^{-1}(v\mathcal B) &\longrightarrow Y(v) \cup Y(vs_r) \\
(\mathcal B_1, \dots, \mathcal B_{r-1}, v\mathcal B) &\longmapsto \mathcal B_{r-1}.
\end{align*}
We will examine the subsets $\Im(\xi) \cap Y(v)$ and $\Im(\xi) \cap Y(vs_r)$. We will show that
\begin{enumerate}
\item these subsets are affine spaces (either empty, a point, or $\mathbb A^1$); one of them, denoted by $\mathbb A_1$, is closed in $\Im(\xi)$, and the other, denoted by $\mathbb A_2$, is nonempty, open, and dense in $\Im(\xi)$;
\item if $\mathbb A_i \neq \emptyset$, then $\mathbb A_i$ belongs to $\Im(m')$; furthermore, $\xi^{-1}(\mathbb A_i) \cong m'^{-1}(\mathbb A_i)$ under the obvious identification, and $\xi\colon \xi^{-1}(\mathbb A_i) \rightarrow \mathbb A_i$ corresponds to the morphism $m'\colon m'^{-1}(\mathbb A_i) \rightarrow \mathbb A_i$.
\end{enumerate}
These facts are enough to prove Theorem~\ref{thmC}. Indeed, applying $\xi^{-1}$ to the decomposition
$$
\Im(\xi) = \mathbb A_1 \cup \mathbb A_2
$$
and using \eqref{sBNpl-2} gives us a decomposition
$$
m^{-1}(v\mathcal B) = \xi^{-1}(\mathbb A_1) \cup \xi^{-1}(\mathbb A_2) = m'^{-1}(\mathbb A_1) \cup m'^{-1}(\mathbb A_2), 
$$
where the first set
is closed and the second is nonempty and open. By the induction hypothesis, the fibers of $m'$ are paved by affine spaces. Since $\mathbb A_i$ is contained in a $\mathcal B$-orbit, we see $m'$ is trivial over each $\mathbb A_i$ by Proposition~\ref{strat_triv}, and hence each $m'^{-1}(\mathbb A_i)$ is paved by affine spaces. Thus $m^{-1}(v\mathcal B)$ is paved by affine spaces. 

To verify  properties \eqref{sBNpl-1} and~\eqref{sBNpl-2}, we need to consider various cases. We start with two cases which arise from the following standard lemma about the Bruhat order.

\begin{lem} \label{trick} Let $(W,S)$ be a Coxeter group and $x,y \in W$ and $s \in S$. Then $x \leq y$ implies $x \leq ys$ or $xs \leq ys$ \textup{(}or both\textup{)}.  
\end{lem}

\begin{proof}
This argument can be found in the literature (see \textit{e.g.} \cite[Proposition~5.9]{Hum}), but we give a short proof for the convenience of the reader.  Suppose $x \leq y$.  We may assume that $ys < y$.  Then we may write $x = \widetilde{ys}\,\widetilde{s}$, where $\widetilde{ys}$ (resp., $\widetilde{s}$) is a subword of any reduced word giving $ys$ (resp., of $s$).  If $x = \widetilde{ys}$, then $x \leq ys$. If $x = \widetilde{ys} s$, then $xs = \widetilde{ys} \leq ys$.
\end{proof}

Recall $v \leq s_*$ by assumption. The two cases we need to consider are as follows: 

\begin{enumerate}[label={\rm Case \Roman*:}, ref=\Roman*, wide,labelindent=0pt,leftmargin=!]
\item\label{CaseI} 
$s'_* < s'_* s_r$, so that $s_* = s'_* s_r$.  Thus by Lemma~\ref{trick}, $v \leq s'_*$ or $v s_r \leq s'_*$.

\item\label{CaseII}
$s'_* s_r < s'_*$, so that $s_* = s'_*$.  Thus $v \leq s'_*$.
\end{enumerate}

We will break each of these into subcases, depending on whether $v < vs_r$ or $vs_r < v$. We then consider further subcases depending on which of $v$ or $vs_r$ precedes $s'_*$ in the Bruhat order. 

\begin{enumerate}[label={\rm Case I.\arabic*:}, ref=I.\arabic*, wide,labelindent=0pt,leftmargin=!]
\item\label{CaseI1} 
$v < vs_r$. So $v \leq s'_*$ is automatic. There are two subcases:
\begin{enumerate}[label={\rm I.1\alph*:}, ref=I.1\alph*,leftmargin=36pt]
\item\label{CaseI1a}  $v < vs_r \leq s'_*$.
\item\label{CaseI1b}  $v \leq s'_*$ but $vs_r \nleq s'_*$.
\end{enumerate}
\item\label{CaseI2} $vs_r < v$. So $vs_r \leq s'_*$ is automatic. There are two subcases:
\begin{enumerate}[label={\rm I.2\alph*:}, ref=I.2\alph*,leftmargin=36pt]
\item\label{CaseI2a} $vs_r < v \leq s'_*$.
\item\label{CaseI2b} $vs_r \leq s'_*$ but $v \nleq s'_*$.
\end{enumerate}
\end{enumerate}
\begin{enumerate}[label={\rm Case II.\arabic*:}, ref=II.\arabic*, wide,labelindent=0pt,leftmargin=!]
\item\label{CaseII1} 
$v < vs_r$.  As $v \leq s'_*$ is automatic, there are two subcases:
\begin{enumerate}[label={\rm II.1\alph*:}, ref=II.1\alph*,leftmargin=36pt]
\item\label{CaseII1a}  $v < vs_r \leq s'_*$.
\item\label{CaseII1b}  $v \leq s'_*$ but $vs_r \nleq s'_*$.
\end{enumerate}
\item\label{CaseII2} 
  $vs_r < v$. Here $vs_r \leq s'_*$ and $v \leq s'_*$, so there are no further subcases.
\end{enumerate}

Consider any element $(\mathcal B_1, \dots, \mathcal B_{r-1}, v\mathcal B)$ in $m^{-1}(v\mathcal B)$.  As already noted above, we have $\mathcal B \, \overset{v}{\longdash} \,  v\mathcal B \, \overset{\leq s_r}{\longdash} \, \mathcal B_{r-1}$.  Then Lemma~\ref{BN_pair_lem} tells us the shape of $\Im(\xi) \cap Y(v)$ and $\Im(\xi) \cap Y(vs_r)$ in all the cases enumerated above.  We record the results in the following table: 

\bigskip

{\small\sf
  \begin{center}
    {\renewcommand{\arraystretch}{1.1}
\begin{tabular}{|c|c|c|}
\hline
Case & $\Im(\xi) \cap Y(v)$ & $\Im(\xi) \cap Y(vs_r)$ \\
\hline \hline

\ref{CaseI1a} & $\mathbb A^0$ & $\mathbb A^1$ \\ \hline
\ref{CaseI1b} & $\mathbb A^0$ & $\emptyset$ \\ \hline
\ref{CaseI2a} & $\mathbb A^1$ & $\mathbb A^0$ \\ \hline
\ref{CaseI2b} & $\emptyset$ & $\mathbb A^0$ \\ \hline
\ref{CaseII1a} & $\mathbb A^0$ & $\mathbb A^1$ \\ \hline
\ref{CaseII1b} & $\mathbb A^0$ & $\emptyset$ \\ \hline
\ref{CaseII2} & $\mathbb A^1$ & $\mathbb A^0$ \\ \hline
\end{tabular}}
\end{center}}
\vskip.3cm

In each case it is clear which piece should be labeled $\mathbb A_1$ or $\mathbb A_2$. This proves the main part of \eqref{sBNpl-1} and~\eqref{sBNpl-2}; the other assertions are clear. This completes the proof of Theorem~\ref{thmC}.
\qed

\begin{rmk}\label{A1_rem} Each $\mathbb A^1$ appearing in the table may be identified with a suitable affine root group $U_{\alpha +n}$, the $k$-group with $k$-points
$$
U_{\alpha + n}(k) = \{ u_\alpha(x t^n) \, | \, x \in k\},
$$
where $u_\alpha \colon \mathbb G_a \rightarrow G$ is the root homomorphism corresponding to the root $\alpha$. For example, consider  Case~\ref{CaseI1a}. Then $\Im(\xi) \cap Y(vs_r)$ is $\{\calB_{r-1} \, | \, v\calB \overset{s_r}{\longdash} \calB_{r-1}\}$. Each such $\calB_{r-1}$ can be expressed as $\calB_{r-1} = {v u s_r}\calB$ for a unique $u \in \mathcal U \cap \, ^{s_r}\overline{\mathcal U}_\calB$.  Now use Proposition~\ref{dHL.3.7.4}\eqref{dHL.3.7.4b}.
\end{rmk}

\subsection{Proof Theorem~\ref{thmA} in a special case} \label{ThmA_spec_case_subsec}

We will now prove Theorem~\ref{thmA} in the case where $\calP=\calB$ and $w_i = s_i$ is a simple reflection for all $1 \leq i \leq r$. The argument is by induction on $r$, as in the previous subsection.  We consider the analogues $p$ and $p'$ of the morphisms $m$ and $m'$
\begin{align*}
p\colon Y(s_\bullet) &\longrightarrow X(s_*), \quad (\calB_1, \dots, \calB_{r-1}, \calB_r) \longmapsto \calB_{r},  \\
p'\colon Y(s'_\bullet) &\longrightarrow X(s'_*),\quad (\calB_1, \dots, \calB_{r-2}, \calB_{r-1}) \longmapsto \calB_{r-1}
\end{align*}
and for $v\calB$ in the image of $p$, we consider the map
\begin{align*}
\xi^\circ\colon p^{-1}(v\mathcal B) &\longrightarrow Y(v) \cup Y(vs_r) \\
(\mathcal B_1, \dots, \mathcal B_{r-1}, v\mathcal B) &\longmapsto \mathcal B_{r-1}.
\end{align*}
The locally triviality of $p$ over $\calB$-orbits in its image still holds, and similarly for $p'$ (see the proof of Proposition~\ref{strat_triv}), and it suffices to establish the analogues of \eqref{sBNpl-1} and~\eqref{sBNpl-2} above.  We consider the same cases as above, and we list the possibilities for $\Im(\xi^\circ) \cap Y(v)$ and $\Im(\xi^\circ) \cap Y(vs_r)$ in the table below, determined in each case with the help of Lemma~\ref{strict-BN_pair_lem}.

\bigskip

{\small\sf
\begin{center}   {\renewcommand{\arraystretch}{1.1}
\begin{tabular}{|c|c|c|}
\hline
Case & $\Im(\xi^\circ) \cap Y(v)$ & $\Im(\xi^\circ) \cap Y(vs_r)$ \\
\hline \hline

\ref{CaseI1a} & $\emptyset$ & \mbox{$\mathbb A^1$ or $\emptyset$} \\ \hline
\ref{CaseI1b} & $\emptyset$ & $\emptyset$ \\ \hline
\ref{CaseI2a} & \mbox{$\mathbb A^1 - \mathbb A^0$ ~or~ $\emptyset$} & \mbox{$\mathbb A^0$ or $\emptyset$} \\ \hline
\ref{CaseI2b} & $\emptyset$ & \mbox{$\mathbb A^0$ or $\emptyset$} \\ \hline
\ref{CaseII1a} & $\emptyset$ & \mbox{$\mathbb A^1$ or $\emptyset$} \\ \hline
\ref{CaseII1b} & $\emptyset$ & $\emptyset$ \\ \hline
\ref{CaseII2} &  \mbox{$\mathbb A^1 - \mathbb A^0$ ~or~ $\emptyset$} & \mbox{$\mathbb A^0$ or $\emptyset$} \\ \hline
\end{tabular}}
\end{center}}
\vskip.3cm
Let us explain the meaning of entries such as ``$\mbox{$\mathbb A^1 - \mathbb A^0$ ~or~ $\emptyset$}$,'' for example in the entry in Case~\ref{CaseI2a} for $\Im(\xi^\circ) \cap Y(v)$.  Note that $v \leq s'_*$ implies that $Y(v) \subset \Im(m')$, but $Y(v) \subset \Im(p')$ is not automatic. However, since $p'$ is $\calB$-equivariant, we have either $Y(v) \cap \Im(p') = \emptyset$ or $Y(v) \subset \Im(p')$.  If $Y(v) \cap \Im(p') = \emptyset$, the table entry is $\emptyset$. If $Y(v) \subset \Im(p')$, the intersection $\Im(\xi^\circ) \cap Y(v)$ is precisely the part of $Y(v)$ which is exactly of relative position $s_r$ from $v\calB$, and this identifies with $\mathbb A^1 - \mathbb A^0$ in the case where $vs_r < v$.

The analogues of  \eqref{sBNpl-1}, \eqref{sBNpl-2} above hold, except that here, both $\mathbb A_1$ and $\mathbb A_2$ can be empty, and when nonempty the larger subset can be either $\mathbb A^0$, $\mathbb A^1 - \mathbb A^0$, or $\mathbb A^1$. The morphism $p'$ is trivial over every $\calB$-orbit in its image (\textit{cf.} Proposition~\ref{strat_triv}), and by induction the nonempty fibers of $p'$ are paved by finite products of copies of $\mathbb A^1$ and $\mathbb A^1- \mathbb A^0$ (use Lemma~\ref{triviality_lem}). (Here, as in the end of the proof of Proposition~\ref{strat_triv}, we are implicitly using that the scheme-theoretic fiber product over $k$ of a reduced $k$-scheme with $X$ is still reduced, even if $k$ is not perfect, if $X$ is a $k$-scheme such as $\mathbb A^1$, $\mathbb A^1 - \mathbb A^0$, or $\mathbb A^0$.) Therefore, the fibers of $p$ also have the desired property. This proves Theorem~\ref{thmA} in the special case where $\calP = \calB$ and each $w_i$ is a simple reflection~$s_i$. \qed

\begin{rmk} \label{Gm_rem}
As in Remark~\ref{A1_rem}, each $\mathbb A^1$ in the table may be identified with an affine root group $U_{\alpha + n}$, and each $\mathbb A^ 1- \mathbb A^0$ may be identified with a suitable variety of nonidentity elements $U^*_{\alpha + n}$.  For example, consider Case~\ref{CaseI2a}.  Then $\Im(\xi^\circ) \cap Y(v)$ is $\{\calB_{r-1} \, | \, v\calB \overset{s_r}{\longdash} \calB_{r-1}\} \cap Y(v)$.  We may write such $\calB_{r-1}$ as 
$$
\calB_{r-1} = vus_r\calB
$$
for a unique $u \in \mathcal U \cap \,^{s_r}\overline{\mathcal U}_\calB$ such that $u \neq e$. Now use Proposition~\ref{dHL.3.7.4}\eqref{dHL.3.7.4b}.
\end{rmk}

\begin{rmk} \label{s,s_eg}
In Cases~\ref{CaseI2a} and~\ref{CaseII2}, the $\mathbb A^0$ piece is in the closure of the $\mathbb A^1 - \mathbb A^0$ piece, and it is tempting to consider the union of these as $\mathbb A^1$.  Indeed, if one ignores the possibility of $\emptyset$ in Cases~\ref{CaseI2a} and~\ref{CaseII2}, the table seems to show that in every case $\Im(\xi^\circ)$ is an affine space ($\emptyset$, $\mathbb A^0$, or $\mathbb A^1$), and one could ask whether the argument does not in fact prove (by induction again) that every fiber of $p$ is paved by affine spaces.  However, one cannot ignore the empty set, and in fact in Case~\ref{CaseII2} it is possible to have $\Im(\xi^\circ) \cap Y(v) = \mathbb A^1 - \mathbb A^0$ while $\Im(\xi^\circ) \cap Y(vs_r) = \emptyset$. Letting $s \in S_{\aff}$, this happens for $s_\bullet = (s_1, s_2) = (s,s)$ and $v = s_2= s$.  This situation is reflected by the quadratic relation in the Iwahori--Hecke algebra $T_s * T_s = (q-1)T_s + qT_1$. In addition, even in a special situation where $\Im(\xi^\circ)$ is always an affine space, the affine space paving would remain elusive, as it is not clear that $p'$ would be trivial over all of $\Im(\xi^\circ)$ whenever it is not contained in a single $\calB$-orbit.
\end{rmk}

\begin{rmk} \label{KLM_eg}
Remark~\ref{s,s_eg} ``explains'' why we cannot hope to improve Theorem~\ref{thmA} to assert that all fibers of $Y_\calP(w_\bullet) \rightarrow X_\calP(w_*)$ are paved by {\em affine spaces}. For a concrete example related to the affine Grassmannian $\Gr_G = LG/L^+P_{\bf 0}$ over a finite field $k = \mathbb F_q$, take $G = {\rm SO}(5)$, and let  
$$
\mu_1 = \mu_2 = \mu_3 = \alpha^\vee_1 + \alpha^\vee_2 = (1,1),
$$
where the $\alpha^\vee_i$ are the two simple coroots of $G$. Here we use notation following the conventions of \cite{Bou}.  In \cite[Section~8.5]{KLM}, it is shown that the Hecke algebra structure constant $c^0_{\mu_\bullet}(q)$ (the coefficient of the unit element in the product $1_{K t^{\mu_1}K} * 1_{K t^{\mu_2}K} * 1_{K t^{\mu_3}K}$ for $K = L^+P_{\bf 0}({\mathbb F}_q)$) satisfies $c^0_{\mu_\bullet}(q) = q^5 - q$.  This shows that the fiber over the base point $e_0$ of $Y_{L^+P_{\bf 0}}(\mu_\bullet) \rightarrow X_{L^+P_{\bf 0}}(|\mu_\bullet|)$ cannot be paved by affine spaces over $\mathbb F_q$.
\end{rmk}

\section{Proof of Theorem~\ref{thmA}} \label{thmA_sec}

\subsection{Schubert cells in $\boldsymbol{\Flag_\calB}$ as convolution spaces}  

If $\tau s_1 \cdots s_r = w$ is a reduced expression, we sometimes write $Y(\tau s_1 \cdots s_r)$ for $Y(w)$.  This Schubert cell has the following well-known moduli description of its $k$-points.  

\begin{lem}  \label{unwind} Fix the reduced expression $w = \tau s_1 \cdots s_r$ as above.
\begin{a-enumerate}
\item Giving a $k$-point of\, $Y(w)$ is equivalent to giving a $k$-point of $\tau^{-1}Y(w)$, which is equivalent to giving a sequence of Iwahori subgroups $(\calB_0, \calB_1, \dots, \calB_r)$ such that
$$
\calB =: \calB_0\, \overset{s_1}{\longdash} \, \calB_1 \, \overset{s_2}{\longdash} \, \calB_2 \, \overset{s_3}{\longdash} \,\cdots \, \overset{s_r}{\longdash} \,\calB_r.
$$
\item For any element $y \in LG(k)$, giving a $k$-point of $y^{-1}Y(w)$ is equivalent to giving a sequence of Iwahori subgroups $(\calB_0, \calB_1, \cdots, \calB_r)$ such that
$$
^{y^{-1}}\calB =: \calB_0\, \overset{s_1}{\longdash} \, \calB_1 \, \overset{s_2}{\longdash} \, \calB_2 \, \overset{s_3}{\longdash} \,\cdots \, \overset{s_r}{\longdash} \,\calB_r.
$$
\end{a-enumerate}
\end{lem}

\begin{proof}
In both cases, note that $\tau$ normalizes the Iwahori $\calB$.
\end{proof}

\subsection{Proof of Theorem~\ref{thmA} for $\boldsymbol{\calP = \calB}$} \label{ThmA_P=B_proof}

Consider the morphism $p_{w_\bullet, \calB}\colon Y_\calB(w_\bullet) \rightarrow X_\calB(w_*)$.  For each $1 \leq i \leq r$, we choose a reduced expression
$$
w_i = \tau_i s_{i1} \cdots s_{i n_i}
$$
for $s_{ij} \in S_{\aff}$ and $\tau_i \in \Omega$.  Since conjugation by $\tau_i$ normalizes $\calB$, permutes $S_{\aff}$, and preserves the Demazure product, we may reduce the study of fibers to the case where each $\tau_i$ is $1$.  Then we have
$$
w_* = s_{11} * \cdots * s_{1 n_1} * s_{21} * \cdots * s_{2 n_2}* \cdots \cdots * s_{r1} * \cdots * s_{r n_r} =: s_{**}.
$$
By Lemma~\ref{unwind}, the morphism $p_{w_\bullet, \calB}$ is identified with the morphism $p_{s_{\bullet \bullet}, \calB} \colon Y(s_{\bullet \bullet}) \rightarrow  X(s_{**})$. By Section~\ref{ThmA_spec_case_subsec}, its fibers possess the required pavings. \qed

\subsection{Proof of Theorem~\ref{thmA} in general} \label{ThmA_general}

Let $\mathcal C$ be the class of $k$-varieties which are finite products of copies of $\mathbb A^1$ and $\mathbb A^1 - \mathbb A^0$.

We consider the morphism $p = p_{w_\bullet, \calP} \colon Y_\calP(w_\bullet) \rightarrow X_\calP(w_*)$ and suppose $v\calP$ lies in the image. We prove that the reduced fiber $p^{-1}(v\calP)$ has a $\mathcal C$-paving by induction on $r$.  As before, consider the morphism $p' \colon Y_\calP(w_1, \dots, w_{r-1}) \rightarrow \Flag_\calP$ given by $(\calP_1,\calP_2, \dots, \calP_{r-1}) \mapsto \calP_{r-1}$, and, by a slight abuse, its restriction $\xi^\circ = p'|_{p^{-1}(v\calP)} \colon p^{-1}(v\calP) \rightarrow \Flag_\calP$, defined by $(\calP_1, \dots, \calP_{r-1}, v\calP) \mapsto \calP_{r-1}$. We have
$$
\Im(\xi^\circ) = \Im(p') \cap vY_\calP\left(w_{r}^{-1}\right).
$$
We claim that for any $y \in W$ with corresponding $\calB$-orbit $Y_{\calB \calP}(y)$, the reduced intersection $$\Im(\xi^\circ) \cap Y_{\calB \calP}(y)$$ either is empty or has a $\mathcal C$-paving.  Then since such locally closed subsets cover $\Im(\xi^\circ)$ and since $p'$ is trivial over each such subset, the $\mathcal C$-paving of the reduced fiber $p^{-1}(v\calP)$ will follow by our induction hypothesis applied to $p'$ (using Lemma~\ref{triviality_lem} and Proposition~\ref{strat_triv}). Note that if $\Im(p') \cap Y_{\calB \calP}(y)$ is nonempty, then $Y_{\calB \calP}(y) \subset \Im(p')$, and we are trying to produce a $\mathcal C$-paving of the reduced intersection
$$
Y_{\calB\calP}(y) \cap vY_{\calP}(w_r^{-1}).
$$
 We can pass to $\calB$-orbits by writing $W_\calP w_r^{-1} W_\calP = \bigsqcup_{\eta_m} \eta_m W_\calP$ for $\eta_m \in W$ a finite collection of right-$\bf f$-minimal elements. We then have a locally closed decomposition $Y_{\calP}(w_r^{-1}) = \bigsqcup_{\eta_m} Y_{\calB \calP}(\eta_m)$. Thus we need to show that each reduced intersection
\begin{equation} \label{P_intersection}
Y_{\calB \calP}(y) \cap vY_{\calB \calP}(\eta_m)
\end{equation}
has a $\mathcal C$-paving. We may assume $y$ is also right-$\bf f$-minimal. It is tempting here to assert that by Lemma~\ref{BP_comp}, this is isomorphic to 
\begin{equation} \label{B_intersection}
Y_{\calB }(y) \cap vY_{\calB }(\eta_m).
\end{equation}
However, examples show that (\ref{B_intersection}) can be empty even when (\ref{P_intersection}) is nonempty.  The correct statement is the following.

\begin{lem} \label{saving_lem}
Let $\pi\colon \Flag_{\calB} \rightarrow \Flag_{\calP}$ denote the natural projection morphism.  Suppose $y, z \in W$ are right-${\bf f}$-minimal, and $v \in W$ is any element. Then $\pi$ induces an isomorphism of\, $k$-schemes {\em when they are given reduced subscheme structures}
$$
\pi_{y, zW_{\bf f}} \colon Y_{\calB}(y) \cap \left(\bigsqcup_{w \in W_{\bf f}} vY_{\calB}(zw)\right) ~~ \overset{\lowsim}{\longrightarrow} ~~ Y_{\calB \calP}(y) \cap v Y_{\calB \calP}(z).
$$
\end{lem}

\begin{proof}
By restricting $\pi$ to the full $\pi$-preimage of the right-hand side, we obtain a proper surjective morphism
$$
\pi_{yW_{\bf f}, zW_{\bf f}} \colon \left(\bigsqcup_{w' \in W_{\bf f}} Y_{\calB}(yw')\right) \cap \left(\bigsqcup_{w \in W_{\bf f}} vY_{\calB}(zw)\right) ~~\longrightarrow ~~ Y_{\calB \calP}(y) \cap v Y_{\calB \calP}(z).
$$
Since $Y_{\calB}(y)$ is closed in $\bigsqcup_{w' \in W_{\bf f}} Y_{\calB}(yw')$, the domain of $\pi_{y, zW_{\bf f}}$ is closed in the domain of $\pi_{yW_{\bf f}, zW_{\bf f}}$, and hence $\pi_{y, zW_{\bf f}}$ is proper.  On the other hand, $\pi_{y, zW_{\bf f}}$ is a monomorphism and is surjective; both of these statements follow from Lemma~\ref{BP_comp}.  (Take a point $x \in  Y_{\calB \calP}(y) \cap v Y_{\calB \calP}(z)$; by Lemma~\ref{BP_comp}, we may write $\pi(\tilde{x}) = x$ for a unique element $\tilde{x} \in Y_{\calB}(y)$. Then note that automatically this $\tilde{x}$ also lies in the preimage $\pi^{-1}(vY_{\calB \calP}(z))$.) Since $\pi_{y, zW_{\bf f}}$ is a proper surjective monomorphism, it is an isomorphism \emph{on reduced subscheme structures}.
\end{proof}

We now apply this to (\ref{P_intersection}) with $z = \eta_m$.  It remains to prove that reduced intersections of the form
$$
Y_\calB(y) \cap v Y_\calB(\eta_m w)
$$
admit $\mathcal C$-pavings, for any element $w \in W_{\bf f}$.

This reduced intersection is turn is equal to the reduced fiber over $v\calB$ of the morphism 
$$
Y_\calB(y) \widetilde{\times} Y_\calB\left((\eta_m w)^{-1}\right) \longrightarrow X_\calB\left(y * (\eta_m w)^{-1}\right),
$$
by the appropriate special case of Lemma~\ref{key_fiber_lem} below. But each fiber of this morphism has a $\mathcal C$-paving by Section~\ref{ThmA_P=B_proof}. \qed

\begin{lem} \label{key_fiber_lem}
For any $w_1, w_2 \in W_\calP \backslash W/ W_\calP$ and $v \in W$, we have an isomorphism of\, $k$-schemes \textup{(}not necessarily given reduced structure\textup{)}
\begin{equation} \label{key_fiber_eq}
Y_\calP(w_1) \cap v Y_\calP\left(w_2^{-1}\right) \cong  p_{w_\bullet}^{-1}(v\calP).
\end{equation}
\end{lem}

\begin{proof}
Each side is a scheme which is determined by the operation of \'etale sheafification of a certain presheaf on the category ${\Affine}_k$. As sheafification commutes with finite limits, it is enough to prove that the corresponding presheaves are isomorphic.  Suppose $R$ is any $k$-algebra.  A section of the presheaf fiber $p_{w_\bullet}^{-1}(v\calP)(R)$ is a tuple of the form $(\calP_R, g\calP_R, v\calP_R)$, where $\calP_R = L^+\calG_{\bf f}(R)$ and where $g \in \calP_R w_1 \calP_R$ and $g^{-1}v \in \calP_R w_2 \calP_R$.  The means precisely that $g\calP_R \in Y_\calP(w_1)(R) \cap vY_\calP(w_2^{-1})(R)$.  Thus the two presheaves coincide.
\end{proof}

\section{Proof of Corollary~\ref{corB}}  \label{corB_sec}

This follows immediately from Theorem~\ref{thmA}, as we have a decomposition into locally closed subvarieties
\begin{equation} \label{conv_cl_rel}
X_\calP(w_\bullet) = \bigsqcup_{v_\bullet} Y_\calP(v_\bullet), 
\end{equation}
where $v_\bullet$ ranges over all tuples $(v_1, v_2, \dots, v_r) \in W_\calP \backslash W/W_\calP$ such that $v_i \leq w_i$ in the Bruhat order on $W_\calP \backslash W/W_\calP$ for all $i$.  Thus the fiber has a corresponding decomposition, and the result follows from Theorem~\ref{thmA}. \qed

\section{Application to structure constants for parahoric Hecke algebras} \label{HA_app}

Fix a nonarchimedean field $F$ with ring of integers $\mathcal O_F$ and residue field $k_F = \mathbb F_q$.  Let us suppose $G$ is a split group over $\mathbb Z$, and fix a Borel pair $B \supset T$ in $G$, also split and defined over $\mathbb Z$. This gives rise to the extended affine Weyl group $W$ defined using $G \supset B \supset T$ (it agrees with the extended affine Weyl group attached to $G_F \supset B_F \supset T_F$). For any parahoric subgroup $\calP \subset G(F)$, consider the parahoric Hecke algebra $\mathcal H(G(F)/\!/\calP) = C_c(\calP \backslash G(F) /\calP, \mathbb C)$, give the structure of a unital associative $\mathbb C$-algebra with convolution $*$ defined using the Haar measure on $G(F)$ giving $\calP$ volume $1$. Consider the $\mathbb C$-basis of characteristic functions
$f_w := 1_{\mathcal P w \mathcal P}$ indexed by elements $w \in W_\calP \backslash W/W_\calP$.  We can represent such cosets by maximal length elements $w \in \,^{\bf f} W ^{\bf f}$.

\begin{prop}  \label{str_const_prop}
For any $w_1, w_2 \in \,^{\bf f} W ^{\bf f}$, we have 
$$
f_{w_1} * f_{w_2} = \sum_{v \in \,^{\bf f} W ^{\bf f}} c^v_{w_1, w_2}(q) \, f_v, 
$$
where the structure constant is a nonnegative integer of the form 
$$
c^v_{w_1, w_2}(q) = \sum_{a,b \in \mathbb Z_{\geq 0}} m_{a,b} ~q^a (q-1)^b
$$
for certain nonnegative integers $m_{a,b}$ which vanish for all but finitely many pairs $(a,b)$.
\end{prop}

\begin{proof}
The combinatorics of parahoric Hecke algebras over characteristic zero local fields $F$ are the same as those for $F=\mathbb F_q(\!(t)\!)$ (the parahoric subgroups in each setting chosen to correspond to each other in the obvious way, suitably identifying apartments for $G_F \supset T_F$ and $G_{\mathbb F_q(\!(t)\!)} \supset T_{\mathbb F_q(\!(t)\!)}$ and facets therein---for a much more general statement, see \cite[Section~4.1.2]{PZ13}).  Therefore, we can assume $F$ is of the latter form. Then note that $c^v_{w_1, w_2}(q)$ is the number of $\mathbb F_q$-rational points in the fiber over $v\calP$ of the corresponding convolution morphism $Y_\calP(w_1) \widetilde{\times} Y_\calP(w_2) \rightarrow X_\calP(w_*)$.  Thus the result follows from Theorem~\ref{thmA}.
\end{proof}

This gives rise to general parahoric variants (in the equal parameter case) of combinatorial results on structure constants for spherical affine Hecke algebras due to Parkinson \cite[Theorem~7.2]{Par06} and Schwer \cite{Schw06}. By virtue of the Macdonald formula (see \textit{e.g.} \cite[Theorem~5.6.1]{HKP}), the function $P_\lambda$ considered (albeit with differing normalizations) by Parkinson and Schwer agrees up to an explicit normalizing factor with the Satake transform $f^\vee_\lambda$ of the basis elements $f_{\lambda} = 1_{G(\mathbb F_q[\![t]\!]) t^\lambda G(\mathbb F_q[\![t]\!])}$ above, for any dominant $\lambda \in X_*(T)$. In particular, Proposition~\ref{str_const_prop} shows that suitably renormalized versions of the functions $C^\nu_{\lambda \mu}$ appearing in \cite[Theorem~1.3]{Schw06} lie in $\mathbb Z_{\geq 0}[q-1]$.

\section{Cellular paving of certain subvarieties in the affine Grassmannian} \label{KU_GrG_sec}

In this section, we will restrict our attention to certain generalizations of the intersections containing the Mirkovic--Vilonen cycles in the affine Grassmannian.  Let $\calP = \calP_{\bf 0} = L^+G$, and consider the affine Grassmannian $\Gr_G = \Flag_\calP$.  We fix any standard parabolic subgroup $P \supset B$ with Levi factorization $P = MN$ for a Levi subgroup $M \supset T$ and unipotent radical $N \subset U$.  Here $ B = TU$ is the Levi decomposition of the fixed Borel subgroup $B$.

We abbreviate $K = L^+G$ and note that the intersection $K_M := K \cap LM$ in $LG$ can be identified with $L^+M$.  We define $K_P := K_M \cdot LN$.  This is a semidirect group ind-scheme over $k$ since $K_M$ normalizes $LN$. For $\lambda \in X_*(T)$, denote the corresponding point by $x_\lambda := \lambda(t) e_{L^+G} \in \Gr_G(k)$. 

Fix $\mu \in X_*(T)^+$.  Recall \cite[Definition~3.1]{HKM}, in which we declare $\nu \in X_*(T)$ satisfies $\nu \geq^P \mu$ provided that
\begin{itemize}
\item $\langle \alpha, \nu \rangle = 0$ for all $T$-roots $\alpha$ appearing in $\Lie(M)$;
\item $\langle \alpha, \nu + \lambda \rangle > 0$ for all $T$-roots $\alpha$ appearing in $\Lie(N)$ and for all $\lambda \in \Omega(\mu)$.
\end{itemize}
Here $\Omega(\mu) = \{ \lambda \in X_*(T) \, | \, \mu - w\lambda ~\mbox{is a sum of positive coroots, for all $w \in W_0$}\}$. Also, let $X_*(T)^{+_M}$ be the cocharacters which are dominant for the roots appearing in $\Lie(B \cap M)$.

\begin{prop} If $\nu \geq^P \mu$ for $\mu \in X_*(T)^+$, and if $\lambda \in \Omega(\mu) \cap X_*(T)^{+_M}$, then there is an equality of\, $k$-subvarieties in $\Gr_G$ 
\begin{equation} \label{HKM_eq}
\left(t^{-\nu}Kt^\nu\right)x_\lambda \,\cap\, K x_\mu = K_Px_\lambda \,\cap\, K x_\mu.
\end{equation}
\end{prop}

\begin{proof}
The equality $(t^{-\nu}Kt^\nu)x_\lambda \,\cap\, \overline{K x_\mu} = K_Px_\lambda \,\cap\, \overline{K x_\mu}$ follows on combining \cite[Proposition~7.1]{HKM} and \cite[Lemma~7.3]{HKM}. The desired equality without the closures follows formally from this one.
\end{proof}

The left-hand side of (\ref{HKM_eq}) admits a cellular paving by Theorem~\ref{thmA}. Indeed, we have an equality of reduced subschemes
$$
\left(t^{-\nu}Kt^\nu\right)x_\lambda \,\cap\, K x_\mu = p_{w_\bullet, L^+G}^{-1} \left(t^{-\nu} e_{L^+G}\right),
$$
for $w_\bullet = (t_\mu, t_{-\nu-\lambda})$; see Lemma~\ref{key_fiber_lem}. Hence we deduce the following result.

\begin{cor}\label{MN_cor} For $\mu, \lambda$ as above, the variety $L^+M\,LN\,x_\lambda \cap L^+G\,x_\mu$ in $\Gr_G$ admits a cellular paving. In particular, for $P = B$, the Mirkovic--Vilonen variety $LUx_\lambda \cap L^+G x_\mu$ admits a cellular paving. 
\end{cor}

Note that this applies to all pairs $(\mu, \lambda) \in X_*(T)^+ \times X_*(T)^{+_M}$: If the intersection is nonempty, then $\lambda \in \Omega(\mu)$ is automatic, by \cite[Lemma~7.2(b)]{HKM}.

\section{Paving results over \texorpdfstring{$\boldsymbol{\mathbb Z}$}{Z}} \label{Z_sec}

The goal of what follows is to extend the constructions and results above to work over $\mathbb Z$.  Because there is no building attached to a group over $\bbZ\pot{t}$, the main challenge is to give purely group-theoretic arguments for certain results which are usually proved with the aid of buildings.

\subsection{Basic constructions over $\boldsymbol{\mathbb Z}$}

We shall recall the basic notions attached to groups over $\mathbb Z$.  One useful reference is \cite[Section~4]{RS20}, but in places we have chosen a slightly different way to justify the foundational results (for example, we do not assume the existence of the Demazure resolutions over $\mathbb Z$---a result stated without proof in \cite{Fal03}---and instead we construct them as a special case of the convolution morphisms over $\bbZ$). 

We assume $G$ is a reductive group over $\mathbb Z$, more precisely, a smooth affine group scheme over $\mathbb Z$ whose geometric fibers are connected reductive groups, and which admits a maximal torus $T$ over $\mathbb Z$, which is automatically split (see \cite[Section~6.4, Example~5.1.4]{Co14}).
We fix a Borel pair over $\mathbb Z$, given by $G \supset B \supset T$ (Borel subgroups $B \supset T$ exist, by \textit{e.g.} \cite[Proof of Theorem~5.1.13]{Co14}).  Following \cite[Section~4]{RS20}, we have the usual objects: the standard apartment endowed with its Coxeter complex structure given by the affine roots,  the base alcove ${\bf a}$ and other facets ${\bf f}$ therein, the Weyl group $W_0$, the Iwahori--Weyl group $W$, the affine Weyl group $W_{\aff}$, and the stabilizer subgroups $W_{\bf f} \subset W_{\aff}$. The Iwahori--Weyl group $W := N_G(T)(\mathbb Z(\!(t)\!))/T(\mathbb Z[\![t]\!])$ can be identified with the extended affine Weyl group $X_*(T) \rtimes W_0$, where using \cite[Proposition~5.1.6]{Co14} we may identify $W_0 = N_G(T)(\mathbb Z[\![t]\!])/T(\mathbb Z[\![t]\!])$.  As $X_*(T) \rtimes W_0$ remains unchanged upon base changing along $\mathbb Z \rightarrow k$ for any field $k$, it inherits a Bruhat order $\leq$ as in the classical theory over a field. Similarly, the apartment is canonically identified with the apartments attached to $(G_{\bbQ\rpot{t}}, T_{\bbQ\rpot{t}})$ or $(G_{\bbF_p\rpot{t}}, T_{\bbF_p\rpot{t}})$ for any prime number~$p$.

We define in the obvious way the positive loop group $L^+G_\mathbb Z$ (a pro-smooth affine group scheme over $\mathbb Z$) and the loop group $LG_\mathbb Z$ (an ind-affine group ind-scheme over $\mathbb Z$). For representability, see \textit{e.g.} \cite[Lemma~3.2]{HR20}.

The following result is essentially due to Pappas and Zhu, and this precise form was checked jointly with Timo Richarz. This is a very special case of the general construction due to Louren\c{c}o, in \cite[Section~3]{Lou23}.

 \begin{lem} \label{PZ_ext_lem} Let ${\bf f}$ be any facet of the apartment corresponding to $T$ in the Bruhat--Tits building of $G(\bbQ\rpot{t})$, and let $\calG_{\bf f_\bbQ}$ be the associated parahoric $\bbQ\pot{t}$-group scheme with connected fibers and with generic fiber $G\otimes_\bbZ\bbQ\rpot{t}$. Then there exists a unique smooth affine fiberwise connected $\bbZ\pot{t}$-group scheme $\calG_{\bf f}$ of finite type extending $\calG_{\bf f_\bbQ}$ with the following properties:\smallskip\\
 i\textup{)} There is an identification of\, $\bbZ\rpot{t}$-groups $\calG_{\bf f} \otimes_{\bbZ\pot{t}}\bbZ\rpot{t}=G\otimes_\bbZ\bbZ\rpot{t}$. \smallskip\\ 
 ii\textup{)} For every prime number $p$, the group scheme $\calG_{\bf f}\otimes_{\bbZ\pot{t}}\bbF_p\pot{t}$ is the Bruhat--Tits group scheme with connected fibers for $G\otimes_\bbZ\bbF_p\rpot{t}$ associated with ${\bf f}$.
 \end{lem}
 \begin{proof} This is proven in \cite[Section~4.2.2]{PZ13}. Note that the base ring in {\it loc.~cit.} is the polynomial ring $\calO[t]$, where $\calO$ is discretely valued. The same proof remains valid over the base ring $\bbZ\pot{t}$ using \cite[Section~3.9.4]{BT84}. 
 \end{proof}
 
 For each ${\bf f}$, we define the ``parahoric'' subgroup $L^+\calG_{\bf f} \subset LG_{\mathbb Z}$, and we often abbreviate by writing $\calP_\bbZ := L^+\calG_{\bf f}$.  This has the property that for each homomorphism $\mathbb Z \rightarrow k$ for $k$ a field, we have $\calP_{\bbZ} \otimes_{\bbZ} k \cong \calP_k$, where the latter is the object defined earlier when working over the field $k$.  We define the (partial) affine flag variety
 $$\Flag_{\calP,\bbZ} = \left(LG_\bbZ/ \calP_\bbZ\right)^{\et},$$
 the \'{e}tale sheafification of the quotient presheaf on ${\Affine}_\mathbb Z$. This is represented by an ind-projective ind-scheme over $\mathbb Z$; see \cite[Corollary~3.11]{HR20}, where the proof is given for objects defined over $\mathcal O[t]$ for any noetherian ring $\mathcal O$ -- a similar proof works in our setting over $\bbZ\pot{t}$.

We denote the base point in $\Flag_{\calP,\bbZ}$ by $e_{\calP, \bbZ}$.

We have a notion of a negative parahoric loop group and a corresponding open cell in $\Flag_{\calP, \bbZ}$.  We define $L^{--}G_\bbZ := \ker(L^-G_\bbZ \rightarrow G_\bbZ)$, $t^{-1} \mapsto 0$, where $L^-G_\bbZ(R) = G(R[t^{-1}])$. Following \cite{dHL18}, we define $L^{--}\calG_{\bf a, \bbZ} = L^{--}G_{\bbZ} \rtimes \overline{U}_\bbZ$.  Then for any facet ${\bf f}$ in the closure of ${\bf a}$, we define the negative parahoric loop group
\begin{equation} \label{neg_parahoric_def}
L^{--}\calG_{\bf f,\bbZ} := \bigcap_{w \in W_{\bf f}} \,^w(L^{--}\calG_{\bf a, \bbZ}),
\end{equation}
the intersection being taken in $LG_\bbZ$.

\begin{lem}  \label{big_cell_Z} The multiplication map $L^{--}\calG_{\bf f, \bbZ} \times L^+\calG_{\bf f, \bbZ} ~ \rightarrow ~ LG_\bbZ$ is representable by a quasi-compact open immersion.
\end{lem}

\begin{proof}
This is proved in the same way as \cite[Lemma~3.6]{HLR}, which proves the analogous result when the base ring is a ring of Witt vectors $\bbW$ instead of $\bbZ$; the same argument works for our group schemes $\calG_{\bf f, \bbZ}$ over $D_\bbZ := \bbZ\pot{t}$. We omit the details. A proof in a more general context is due to Louren\c{c}o; see \cite[Corollary~4.2.11]{Lou23}, which also points out that the multiplication morphism is affine (hence automatically quasi-compact). 
\end{proof}

From now on, we often write $\calG$ for $\calG_{\bf f}$ and $\calP_\bbZ$ for $L^+\calG_{\bf f}$.

We recall the interpretation of partial affine flag varieties in terms of suitable spaces of torsors. For any ring $R$, set $D_R = \Spec(R\pot{t})$ and $D^*_R = \Spec(R\rpot{t})$. Recall that we define the sheaf $\Gr_{\calG}$ on ${\Affine}_\bbZ$ to be the functor sending $R$ to the set $\Gr_\calG(R)$ of isomorphism classes of pairs $(\calE, \alpha)$, where $\calE$ is a right \'{e}tale torsor for $\calG_{D_R} = \calG \times_{D_\bbZ} D_R$ over $D_R$, and where $\alpha \in \calE(D^*_R)$, that is, $\alpha$ is an isomorphism of $\calG_{D^*_R}$-torsors $\calE_0|_{D^*_R}  \overset{\lowsim}{\rightarrow} \calE |_{D^*_R} $, where $\calE_0$ is the trivial $\calG_{D_R}$-torsor.   The left action of $g \in LG(R)$ on $\Gr_\calG(R)$ sends $(\calE, \alpha)$ to $(\calE, \alpha \circ g^{-1})$. Then $\Gr_\calG(R) \cong \Flag_{\calP, \bbZ}(R)$, functorially in $R$ (see \textit{e.g.} \cite[Lemma~3.4]{HR20}).

\begin{rmk} \label{Ces_rem}
For the groups $G$ over $\bbZ$ we consider, one can show using negative parahoric loop groups that the morphism $LG_\bbZ \rightarrow \Flag_{\calP, \bbZ}$ has sections locally in the Zariski topology, and hence for any semi-local ring~$R$, we have $\Flag_{\calP, \bbZ}(R) = LG_\bbZ(R)/\calP_\bbZ(R)$.  This can be seen by reducing to the case of fields, as in \cite[Section~4.3]{RS20}.  One can also deduce it from a recent result of \v{C}esnavi\v{c}ius \cite[Theorem~1.7]{Ces22} that the affine Grassmannian $\Gr_{G, \bbZ}$ agrees with the Zariski sheafification of the presheaf quotient $LG_{\bbZ}/L^+G_\bbZ$.  To use this to prove the corresponding result for a general parahoric $\calP_\bbZ$, one first deduces the result for $\calP_\bbZ = \calB_\bbZ$, using the lifting for $\calP_\bbZ = L^+G_\bbZ$ and the fact that the fiber of $\Flag_{\calB, \bbZ} \rightarrow \Gr_{G, \bbZ}$ over the base point is $(G/B)_\bbZ$ and $G \rightarrow (G/B)_\bbZ$ is Zariski-locally trivial.  Then, finally, one uses the topological surjectivity of $\Flag_{\calB, \bbZ} \rightarrow \Flag_{\calP, \bbZ}$ to prove that a cover given by translates of the big cell in the source maps to a cover of translates of the big cell in the target. In fact one can use translates $wL^{--}\calG_{\bf a} e_{\calB, \bbZ}$ for $w \in W$ to cover $\Flag_{\calB, \bbZ}$, thanks to the Birkhoff decomposition of $LG$ over fields (see \cite[Lemma~4]{Fal03}). I am grateful to Thibaud van den Hove for a clarifying discussion about this remark, which we shall not need in the rest of this article. 
\end{rmk}

\begin{lem} \label{new_base_pt}
Fix a ring $R$ and $(\calE, \alpha) \in \Gr_{\calG}(R)$.  Then the presheaf\, $\Gr_{\calG, \calE, \alpha}$ sending $\Spec(R') \rightarrow \Spec(R)$ to the set of isomorphism classes of pairs $(\calE', \alpha')$ consisting of a $\calG_{D_{R'}}$-torsor $\calE' \rightarrow D_{R'}$ and an isomorphism of $\calG_{D^*_{R'}}$-torsors $\alpha'\colon \calE_{D^*_{R'}} \overset{\lowsim}{\rightarrow} \calE'_{ D^*_{R'}}$ is representable by an ind-projective ind-flat ind-scheme over $R$.
\end{lem}

\begin{proof}
If we fix a representative $(\calE, \alpha)$ within its isomorphism class, then the map 
$$(\calE', \alpha') \longmapsto (\calE', \alpha' \circ \alpha)$$ is a well-defined isomorphism of presheaves $\Gr_{\calG, \calE, \alpha} \overset{\lowsim}{\rightarrow} \Gr_{\calG} \times \Spec(R)$.  Now recall that $\Gr_{\calG}$ is ind-flat over $\mathbb Z$ by adapting the proof of \cite[Proposition~8.9]{HLR}, or by reducing to the case $\calP_\bbZ = L^+G_{\bbZ}$ and then invoking \cite[Prop,\,8.8]{HLR}.
\end{proof}

\subsection{Ingredients needed for paving over $\boldsymbol{\mathbb Z}$}

\subsubsection{Iwahori decompositions of $\boldsymbol{\calB_\bbZ}$ and $\boldsymbol{\mathcal U_\bbZ}$}

Our choice of base Iwahori subgroup $\calB_\bbZ$ is compatible with our choice of Borel subgroup $B = TU$ over $\bbZ$ in the following sense: For any algebra $R$, we have
$$
\calB_\bbZ(R) = \left\{ g \in L^+G_\bbZ(R) \, | \, \bar{g} \in B(R) \right\},
$$
where $\bar{g}$ is the image of $g$ under the canonical projection $L^+G_\bbZ(R)\rightarrow G(R)$. We define the pro-unipotent radical $\calU_\bbZ \subset \calB_\bbZ$ by requiring $\calU_\bbZ(R)$ to be the preimage of $U(R)$ under the projection $g \mapsto \bar{g}$.  Let $\calT_\bbZ$ denote the group scheme $\calT_\bbZ = L^+T_\bbZ$. Let $\overline{B} = T\overline{U}$ be the Borel subgroup such that $B \cap \overline{B} = T$. For any integer $m \geq 1$, let $L^{(m)}G_\bbZ(R)$ denote the kernel of the natural homomorphism $L^+G_\bbZ(R) \rightarrow G(R/t^mR)$.  Write $\calT^{(1)}_\bbZ := L^{(1)}T_\bbZ$.

\begin{prop} \label{Iwahori_decomp_Z}
The group schemes $\calB_\bbZ$ and $\calU_\bbZ$ possess Iwahori decompositions with respect to $B = TU$; that is, there are unique factorizations of functors
\begin{align}\label{Iwahori_decomp_Z_eq}
\calB_\bbZ &= \left(\calB_\bbZ \cap L\overline{U}_\bbZ\right) \cdot \calT_\bbZ \cdot \left(\calB_\bbZ \cap LU_\bbZ\right), \\
\calU_\bbZ &= \left(\calU_\bbZ \cap L\overline{U}_\bbZ\right) \cdot \calT^{(1)}_\bbZ \cdot \left(\calB_\bbZ \cap LU_\bbZ\right).
\end{align}
\end{prop}

\begin{proof}
First we note that the uniqueness in the decomposition follows from the uniqueness of the decomposition in the big cell in $\overline{U} \cdot T \cdot U$ in $G$.

We shall prove only the first decomposition (the second is completely similar). Consider $g \in \calB_\bbZ(R)$, with reduction modulo $t$ given by $\bar{g} = \bar{b}$ for some $b \in B(R) \subset L^+B_\bbZ(R)$. Then $g^{(1)} := g b^{-1} \in L^{(1)}G_\bbZ(R)$, and it suffices to show this element lies in 
\begin{equation} \label{level_1_decomp}
\left(\calB_\bbZ \cap L^{(1)}\overline{U}_\bbZ\right) \cdot \calT^{(1)}_\bbZ \cdot  \left(\calB_\bbZ \cap L^{(1)}U_\bbZ\right).
\end{equation}
The filtration $\cdots \subset L^{(m+1)}G_\bbZ \subset L^{(m)}G_\bbZ \subset \cdots \subset L^+G_\bbZ$ has abelian quotients isomorphic to $\Lie(G)_\bbZ = \Lie(\overline{U})_\bbZ \oplus \Lie(T)_\bbZ \oplus \Lie(U)_\bbZ$.  We claim that we can write
$$
g^{(1)} = \lim_{m \rightarrow \infty} \bar{u}_m \cdot t_m \cdot u_m
$$
with $\bar{u}_m, t_m, u_m$ lying in the $R$-points of the appropriate factors of (\ref{level_1_decomp}), and such that the limit converges in the $t$-adic topology. Indeed, decomposing the image modulo $t^2$ of $g^{(1)}$ in terms of the Lie algebra and lifting, we can write
$$
g^{(1)} ~ = ~ \bar{u}^{(1,2)} \cdot g^{(2)} \cdot t^{(1,2)} \cdot u^{(1,2)}, 
$$
where $\bar{u}^{(1,2)} \in L^{(1)}\overline{U}_\bbZ$, $t^{(1,2)} \in L^{(1)}T_\bbZ$, and $u^{(1,2)} \in L^{(1)}U_\bbZ$ and where $g^{(2)} \in L^{(2)}G_\bbZ$.  Here we have used that $\bar{u}^{(1,2)}$ normalizes $L^{(2)}G_\bbZ$.  We then repeat this process with $g^{(2)}$ and get an expression
$$
g^{(1)} =  \left(\bar{u}^{(1,2)}\bar{u}^{(2,3)}\right) \cdot g^{(3)} \cdot \left(t^{(1,2)}t^{(2,3)}\right) \cdot \left(u^{(2,3)} u^{(1,2)}\right), 
$$
where $g^{(3)} \in L^{(3)}G_\bbZ$ and where $?^{(2,3)}$ refers to a component of $g^{(2)}$ lying in an appropriate $L^{(2)}?$ group, with $g^{(3)}$ viewed as an ``error term.'' Here we have used again the normality of $L^{(m)}G_\bbZ$ in $L^+G_\bbZ$, the commutativity of $L^+T_\bbZ$, and the fact that $L^+T_\bbZ$ normalizes each $L^{(m)}U_\bbZ$.  Continuing this, we define the sequences $\bar{u}^{(m-1,m)}$, $t^{(m-1,m)}$, $u^{(m-1, m)}$, and $g^{(m)}$ and then set
\begin{align*}
\bar{u}_m &= \bar{u}^{(1,2)} \cdots \bar{u}^{(m-1,m)}, \\
t_m &= t^{(1,2)} \cdots t^{(m-1,m)}, \\
u_m &= u^{(m-1,m)} \cdots u^{(1,2)}.
\end{align*}
In the $t$-adic topology, these products converge, and the terms $g^{(m)}$ approach the identity element $e \in G$. Hence this proves the claim, and thus the proposition.
\end{proof}

\subsubsection{Proposition~\ref{dHL.3.7.4} over $\bbZ$}

\begin{prop} \label{dHL.3.7.4_Z}
The analogue over $\mathbb Z$ of Proposition~\ref{dHL.3.7.4} holds.
\end{prop}

\begin{proof}
There are only finitely many affine roots $a$ such that $U_{a,\bbZ}$ is contained in $\calU_\bbZ \cap \,^v\overline{\calU}_{\calP, \bbZ}$, namely the finitely many $a$ with $a > 0$ and $v^{-1}a \overset{\bf f}{<} 0$.  By the Iwahori decomposition Proposition~\ref{Iwahori_decomp_Z} and the root group filtrations in $U_\bbZ$ and $\overline{U}_\bbZ$, we easily see that there exist finitely many positive affine roots $a_1, \dots, a_N$ such that
\begin{equation} \label{product}
\calU_\bbZ = U_{a_1, \bbZ} \cdots U_{a_N, \bbZ} \,\, (\calU_\bbZ \cap\, ^v\calP_\bbZ).
\end{equation}
In what follows, we suppress the subscript $\mathbb Z$. Give a total order $\preceq$ to the set of positive affine roots $a_i$  in this list, by letting $a \prec b$ if and only if $a(x_0) < b(x_0)$ for a suitably general point $x_0 \in {\bf a}$.  Let $r_1 \prec r_2 \prec \cdots \prec r_M$ be the totally order subset of the $a_i$ with the property that $v^{-1}r_i \overset{\bf f}{<} 0$.  The root group $U_{r_1}$ appears finitely many times in (\ref{product}).   Starting from the left, we commute the first $U_{r_1}$ to the left past any preceding groups $U_b$.  By the commutator relations (\textit{e.g.} \cite[Equation~(3.6)]{dHL18}), in moving all the $U_{r_1}$ groups all the way to the left, we introduce finitely many additional affine root groups $U_c$ with $r_1 \prec c$.  Then we consider the part of the product which now involves only root groups  of the form $U_{r_2}, \dots, U_{r_M}$ and certain $U_c$ with $v^{-1}c \overset{\bf f}{\geq} 0$.  Then we repeat the above process with $r_2$ in place of $r_1$.  Continuing, we eventually move all the $U_r$ factors with $v^{-1}r \overset{\bf f}{<} 0$  all the way to the left.  We have proved that
\begin{equation} \label{first_decomp}
\mathcal U = \prod_{r} U_r \, (\calU \cap \,^v\calP), 
\end{equation}
where $r$ ranges over the affine roots with $r > 0 $ and $v^{-1}r \overset{\bf f}{<} 0$. We claim that the obvious inclusion $\prod_r U_r \subset \calU \cap  \, ^v\overline{\calU}_{\calP}$ is an equality and the resulting product is a decomposition.  Both statements follow easily using the theory of the big cell, Lemma~\ref{big_cell_Z}.
\end{proof}

\begin{cor} \label{BP_comp_Z}
The analogues over $\bbZ$ of Propositions~\ref{BP_comp} and~\ref{strat_triv} hold. 
\end{cor}

\subsubsection{Schubert cells and Schubert schemes over $\boldsymbol{\bbZ}$}

Fix $w \in W$, and fix a lift $\dot{w} \in N_GT(\bbZ\pot{t})$ of $w$.  We usually suppress the dot from now on since no construction depends on this choice.   The group $\calP_\bbZ$ acts on the left on $\Flag_{\calP, \bbZ}$; we define the {\em Schubert scheme} $X_{\calP, \bbZ}(w) \subset \Flag_{\calP, \bbZ}$ to be the scheme-theoretic image of the morphism
$$
\calP_\bbZ \longrightarrow \Flag_{\calP, \bbZ}, \quad p \longmapsto p \dot{w} e_{\calP, \bbZ}.
$$
Similarly, we can define $X_{\calQ \calP, \bbZ}(w)$ for any parahoric subgroup $\calQ_\bbZ$; in particular, we have $X_{\calB \calP, \bbZ}$.  

We define $Y_{\calP, \bbZ}(w) \subset \Flag_{\calP, \bbZ}$ to be the \'etale sheaf-theoretic image of the morphism of sheaves $\calP_{\bbZ} \rightarrow \Flag_{\calP, \bbZ}$, $p \mapsto p \dot{w}e_\calP$, and as before we define similarly $Y_{\calQ \calP, \bbZ}$ for any parahoric subgroup $\calQ_\bbZ \subset LG_\bbZ$. 

\begin{lem} \label{rep_lem_Z}
  Let $\calP_\bbZ \subset LG$ be the parahoric subgroup fixed above \textup{(}similar statements apply to any $\calQ$-orbits in $\Flag_{\calP, \bbZ}$\textup{)}.
\begin{a-enumerate}
\item\label{rlZ-a} The scheme $X_{\calP, \bbZ}(w)$ is an integral scheme which is projective and faithfully flat over $\Spec(\bbZ)$, and $X_{\calP, \bbZ}(w) \otimes \bbQ = X_{\calP, \bbQ}(w)$.
\item\label{rlZ-b} The morphism $Y_{\calP,\bbZ}(w) \rightarrow \Flag_{\calP, \bbZ}$ of \'etale sheaves factors canonically as
  $$Y_{\calP, \bbZ}(w) \longrightarrow X_{\calP, \bbZ}(w) \longrightarrow \Flag_{\calP, \bbZ},$$
  and the first morphism is represented by a quasi-compact open immersion of schemes.
\item\label{rlZ-c} The scheme $Y_{\calP, \bbZ}(w)$ is smooth over $\Spec(\bbZ)$, and its formation commutes with base change along an arbitrary homomorphism $\mathbb Z \rightarrow R$.
\end{a-enumerate}
\end{lem}

\begin{proof}
  The projectivity in \eqref{rlZ-a} is proved in \cite[Definition~4.3.4 and what follows]{RS20}.
  Part~\eqref{rlZ-b}  can be proved by adapting the argument of \cite[Corollary~3.14]{Ri16b}. Part~\eqref{rlZ-c}  holds since $Y_{\calP, \bbZ}(w)$ is the orbit under a smooth group scheme over~$\mathbb Z$. Hence $Y_{\calP, \bbZ}(w) \otimes \bbQ = Y_{\calP, \bbQ}(w)$.

The formation of the scheme-theoretic image of a quasi-compact morphism commutes with flat base change (see \cite[Lemma~29.25.16]{StaPro}). So the generic fiber of $X_{\calP, \bbZ}(w)$ is the schematic-closure of $Y_{\calP, \bbQ}(w)$ in $\Flag_{\calP, \bbQ}$; that is, $X_{\calP, \bbZ}(w) \otimes \bbQ = X_{\calP, \bbQ}(w)$.  Now the flat closure of the latter in $X_{\calP, \bbZ}$ contains the scheme-theoretic closure of $Y_{\calP, \bbZ}(w)$, which is all of $X_{\calP, \bbZ}(w)$.  This shows that the latter is faithfully flat over $\mathbb Z$. Clearly, $X_{\calP, \bbZ}(w)$ is irreducible since it is the scheme-theoretic image of a morphism with irreducible source.  Moreover, $X_{\calP, \bbZ}(w)$ is reduced since is it the flat closure of a $\mathbb Q$-variety.
\end{proof}

\subsubsection{Reduction to neutral element of $\boldsymbol{\Omega}$}

In the theory over fields $k$, it is easy to see that $\tau \in \Omega$ has the property that $\tau$ normalizes the standard Iwahori subgroup $\calB_k \subset LG_k$ corresponding to the base alcove ${\bf a}$.   We need to know that this remains true over $\bbZ$.

\begin{lem} \label{Omega_Z}
If $\tau \in \Omega$, then $^\tau\calB_\bbZ = \calB_\bbZ$ as subgroups of\, $LG_\bbZ$.
\end{lem}

\begin{proof}
The identification follows by the uniqueness characterization of the group scheme $\calG_{\bf a, \bbZ}$ in Lemma~\ref{PZ_ext_lem} and the fact that it holds after base change to every field $k$.
\end{proof}

\subsubsection{Twisted products over $\boldsymbol{\bbZ}$}

As above we fix $\calP_\bbZ = L^+\calG_{\bf f, \bbZ}$. Again abbreviate $\calG := \calG_{\bf f, \bbZ}$. Fix $r \in \bbN$, and consider the right action of $\calP_\bbZ^r$ on $LG_\bbZ^r$ given by the same formula as (\ref{r-fold_action}).  

\begin{dfn} \label{twisted_Gr}
We define the $r$-fold twisted product
$$
\widetilde\Gr_{\calG} := LG_\bbZ \times^{\calP_\bbZ} LG_\bbZ \times^{\calP_\bbZ} \cdots \times^{\calP_\bbZ} LG_\bbZ/\calP_\bbZ =: \Gr_{\calG} \tilde{\times} \cdots \tilde{\times} \Gr_{\calG}
$$
to be the \'etale quotient sheaf for the presheaf $(LG_\bbZ)^r/(\calP_\bbZ)^r$ defined above.
\end{dfn}

It is clear from the fact that every $\calG$-bundle over $D_R$ is trivializable over $D_{R'}$ for some \'etale ring extension $R \rightarrow R'$ that we can identify $\widetilde\Gr_{\calG}(R)$ with the set of equivalence classes of tuples
$$
(\calE_\bullet, \alpha_\bullet) = (\calE_1, \dots, \calE_r; \alpha_1, \dots, \alpha_r)
$$
such that each $\calE_i$ is a $\mathcal G_{D_R}$-torsor over $D_R$ and the $\alpha_i \colon \calE_{i-1}|_{D^*_R} \overset{\lowsim}{\rightarrow} \calE_i|_{D^*_R}$ are isomorphisms of $\mathcal G_{D^*_R}$-torsors over $D^*_R$ for all $i = 1, \dots, r$ (with the convention that $\calE_0$ is the trivial torsor).

\begin{lem} \label{conv_repble}
The sheaf $LG_\bbZ \times^{\calP_\bbZ} LG_\bbZ \times^{\calP_\bbZ} \cdots \times^{\calP_\bbZ} LG_\bbZ/\calP_\bbZ$ is represented by an ind-proper ind-scheme which is faithfully flat over $\bbZ$. 
\end{lem}

\begin{proof}
We proceed by induction on $r$.  The case $r =1$ is clear: The ind-flatness of $\Gr_{\calG} \rightarrow \Spec(\bbZ)$ is proved by an easy reduction to the case $\calP_\bbZ = L^+G_\bbZ$, which is then handled by \cite[Proposition~8.8]{HLR}.

Now assume $r > 1$ and that the result holds for $r-1$-fold quotients.  The projection onto the first factor gives a morphism
$$
p\colon LG_\bbZ \times^{\calP_\bbZ} LG_\bbZ \times^{\calP_\bbZ} \cdots \times^{\calP_\bbZ} LG_\bbZ/\calP_\bbZ \longrightarrow LG_\bbZ/\calP_\bbZ.
$$
Now the induction hypothesis and the proof of Lemma~\ref{new_base_pt} show that this morphism is representable by an ind-proper ind-flat ind-scheme, and hence the total space is represented by an ind-scheme.

Locally in the \'etale topology on the target, $p$ is locally trivial with flat fiber, hence is flat.  It follows that the source of $p$ is flat over $\bbZ$. 

It remains to prove the source of $p$ is proper over $\bbZ$.  We know that ind-locally in the \'etale topology on the target, $LG \rightarrow \Gr_\calG$ has sections and hence after passing to an \'etale cover $p$ becomes Zariski-locally trivial with fibers which are ind-proper over $\bbZ$,
by using translates of the big cell (see Lemma~\ref{big_cell_Z}). Since properness descends along \'etale covers, we conclude that $p$ is ind-proper, as desired.
\end{proof}

Let $w \in W$. Denoting the quotient morphism by $q\colon LG_\bbZ \rightarrow \Flag_{\calP, \bbZ}$, note that $\calP_\bbZ w \calP_\bbZ = q^{-1}(Y_{\calP, \bbZ}(w))$ (an equality of \'etale subsheaves of $LG_\bbZ$), where by definition $\calP_\bbZ w \calP_\bbZ$ denotes the \'etale sheaf quotient $\calP_\bbZ \times^{w, \calP_\bbZ} \calP_\bbZ$ of $\calP_\bbZ \times \calP_\bbZ$ by the right action of $\calP_\bbZ \cap w \calP_\bbZ w^{-1}$ given by $(p,p')\cdot \delta = (p\delta, w^{-1}\delta^{-1}w p)$.  In this vein, we {\em define} $\widetilde{\calP_\bbZ w \calP_\bbZ} := q^{-1}(X_{\calP, \bbZ}(w))$.

\begin{dfn}  Let $w_\bullet = (w_1, w_2, \dots, w_r) \in W^r$.  We define
\begin{align*}
Y_{\calP, \bbZ}(w_\bullet) &:= \calP_{\bbZ} w_1 \calP_{\bbZ} \times^{\calP_\bbZ} \calP_{\bbZ} w_2 \calP_\bbZ \times^{\calP_\bbZ} \cdots \times^{\calP_\bbZ} \calP_\bbZ w_r \calP_{\bbZ} /\calP_\bbZ  = Y_{\calP, \bbZ}(w_1) \tilde{\times} Y_{\calP, \bbZ}(w_2) \tilde{\times} \cdots \tilde{\times} Y_\calP(w_r), \\
X_{\calP, \bbZ}(w_\bullet) &:= \widetilde{\calP_{\bbZ} w_1 \calP_{\bbZ}} \times^{\calP_\bbZ} \widetilde{\calP_{\bbZ} w_2 \calP_\bbZ} \times^{\calP_\bbZ} \cdots \times^{\calP_\bbZ} \widetilde{\calP_\bbZ w_r \calP_{\bbZ}} /\calP_\bbZ   = X_{\calP, \bbZ}(w_1) \tilde{\times} X_{\calP, \bbZ}(w_2) \tilde{\times} \cdots \tilde{\times} X_\calP(w_r) 
\end{align*}
to be the \'etale quotient sheaves as in Definition~\ref{twisted_Gr}.
\end{dfn}

\begin{lem} The sheaves $X_{\calP,\bbZ}(w_\bullet)$ and $Y_{\calP, \bbZ}(w_\bullet)$ are represented by integral schemes which are of finite type and flat over $\bbZ$.  Moreover, $X_{\calP, \bbZ}(w_\bullet)$ is proper over $\bbZ$.
\end{lem}

\begin{proof}
The proof goes by induction on $r$, in the same manner as the proof of Lemma~\ref{conv_repble}.
\end{proof}

\subsubsection{Demazure morphisms and closure relations over $\boldsymbol{\bbZ}$}

We need to construct the Demazure resolutions over $\mathbb Z$.  This is stated without proof in \cite{Fal03} and is  implicit in some literature (\textit{e.g.} \cite{PR08, RS20}), but we think some extra discussion is needed.  

For $s \in S_{\aff}$, let $\calG_{s, \bbZ} : = \calG_{{\bf f}, \bbZ}$, where ${\bf f}$ is the facet fixed by $s$. Let $\calP_{s, \bbZ} = L^+\calG_{s, \bbZ}$.  We have $\calP_{s, \bbZ} = \calB_{\bbZ} \cup \calB_{\bbZ} s \calB_\bbZ$ as schemes (to show this we use Lemma~\ref{rep_lem_Z}\eqref{rlZ-c} and the fact that the inclusion $\calB_{\bbZ}(R) \cup \calB_{\bbZ} s \calB_\bbZ(R) \hookrightarrow \calP_{s,\bbZ}(R)$ is surjective when $R$ is any field, but we warn that this equality fails for general $R$, in particular for $R = \bbZ$).  We have an identification $\bbP^1_\bbZ = \calP_{s, \bbZ}/\calB_\bbZ$.  Furthermore, the foregoing shows we have an open immersion $\mathbb A^1_\bbZ = \calB_\bbZ s \calB_\bbZ/\calB_\bbZ \hookrightarrow \bbP^1_\bbZ$ with closed complement $\mathbb A^0_\bbZ = \calB_\bbZ/\calB_{\bbZ} \hookrightarrow \bbP^1_{\bbZ}$. The BN-pair relations hold.

\begin{lem} \label{BN_pair_lem_Z}
For any $w \in W$ and $s \in S_{\aff}$, we have equalities of sub-ind-schemes in $LG_\bbZ$
$$
\calB_\bbZ w \calB_\bbZ s \calB_\bbZ = \begin{cases} \calB_\bbZ ws \calB_\bbZ&\text{if }w < ws,  \\
\calB_\bbZ w \calB_\bbZ \cup \calB_\bbZ ws \calB_\bbZ & \text{if }ws < w. \end{cases} 
$$
\end{lem}

\begin{proof} 
Both cases are proved by induction on $\ell(w)$.  The first case follows from the case of fields and Lemma~\ref{rep_lem_Z}\eqref{rlZ-c}.  For the second case, it is enough to prove the result for $w =s$.  But $\calB_\bbZ s \calB_\bbZ s \calB_\bbZ =  \calB_\bbZ s \calB_\bbZ \cup \calB_\bbZ = \calP_{s, \bbZ}$ follows because $\calP_{s, \bbZ}$ is a group subscheme of $LG_\bbZ$ and $s\calB_\bbZ s \not\subset  \calB_\bbZ$.
\end{proof}

Let $w = s_1 \cdots s_r$ be a reduced word in $W$. Consider the Demazure morphism given by projecting to the final coordinate: 
$$
m_{s_\bullet, \bbZ}\colon D(s_\bullet)_\bbZ := \calP_{s_1, \bbZ} \times^{\calB_\bbZ} \calP_{s_2, \bbZ} \times^{\calB_\bbZ} \cdots \times^{\calB_\bbZ} \calP_{s_r, \bbZ}/ \calB_\bbZ \longrightarrow X_{\calB, \bbZ}(w).
$$
The image lies in $X_{\calB, \bbZ}(w)$ by flatness and properness, and by the fact that this holds over $\bbQ$.  By the BN-pair relations, it gives an isomorphism over $Y_{\calB, \bbZ}(w)$.  Furthermore, it implies the closure relations
\begin{equation} \label{cl_rel_Z}
X_{\calP, \bbZ}(w) = \bigsqcup_{v} Y_{\calP, \bbZ}(v), 
\end{equation}
where $v  \in W_\calP \backslash W/W_\calP$ is such that $v \leq w$ in the Bruhat order on $W_\calP \backslash W/W_\calP$.  In particular, we see that $\widetilde{\calP_\bbZ w \calP_\bbZ} = \bigsqcup_v \calP_\bbZ v \calP_\bbZ$. Here and in (\ref{cl_rel_Z}) the union indicates a union of locally closed subschemes, and every subscheme appearing is reduced by construction.

With the existence of Demazure resolutions over $\bbZ$ in hand, one can prove the following result by copying the argument of \cite[Proposition~3.4]{HLR} (Demazure resolutions over $\bbZ$ are used to prove that Schubert varieties attached to simply connected groups over a field $k$ are normal, following the argument in \cite[Section~9]{PR08}).

\begin{cor}
For any field $k$, $(X_{\calP, \bbZ}(w) \otimes_\bbZ k)_{\red} = X_{\calP, k}(w)$. Further, $X_{\calP, \bbZ}(w) \otimes_{\bbZ} k$ is reduced if and only if $X_{\calP, k}(w)$ is normal.
\end{cor}

\subsubsection{Convolution morphisms over $\boldsymbol{\bbZ}$}

Given the BN-pair relations involving subschemes of $LG_\bbZ$, we have the following.

\begin{lem} \label{conv_mor_Z}
For any $w_\bullet = (w_1, w_2, \dots, w_r) \in W^r$ with Demazure product $$w_* := \,^{\bf f}w^{\bf f}_1 * \,^{\bf f}w^{\bf f}_2 * \cdots *\,^{\bf f}w^{\bf f}_r,$$ 
we have the convolution morphisms and uncompactified convolution morphisms over $\bbZ$:
\begin{align*}
m_{w_\bullet, \calP_\bbZ}\colon X_{\calP, \bbZ}(w_\bullet) &\longrightarrow X_{\calP, \bbZ}(w_*),  \\
p_{w_\bullet, \calP_\bbZ}\colon Y_{\calP, \bbZ}(w_\bullet) &\longrightarrow X_{\calP, \bbZ}(w_*).
\end{align*}
\end{lem}

\begin{lem} \label{key_fiber_lem_Z}
The analogues of Lemma~\ref{key_fiber_lem} and Corollary~\ref{MN_cor} hold over $\mathbb Z$.
\end{lem}

\begin{proof}
The proof of Lemma~\ref{key_fiber_lem} carries over. Then using this, and interpreting retractions group-theoretically (see Remarks~\ref{A1_rem} and~\ref{Gm_rem}), we see that the proof of Corollary~\ref{MN_cor} also carries over.
\end{proof}

\subsection{Main results over $\boldsymbol{\mathbb Z}$}

The following theorem gives the $\bbZ$-versions of Corollaries~\ref{corB} and~\ref{MN_cor}.

\begin{thm} \label{CvdHS_gen} In the notation above, for any $v \in W$, the reduced fiber $m_{w_\bullet, \calP_\bbZ}^{-1}(v\,e_{\calP_\bbZ})$ has a cellular paving over $\bbZ$; that is, it is paved by finite products of $\bbA^1_\bbZ$ and $\bbA^1_\bbZ - \bbA^0_\bbZ$. Further, for every standard parabolic subgroup $P_\bbZ = M_\bbZ N_\bbZ \subset G_\bbZ$ and every pair of cocharacters $(\mu, \lambda) \in X_*(T)^+ \times X_*(T)^{+_M}$, the reduced intersection $L^+M_\bbZ\,LN_\bbZ\,x_\lambda \cap L^+G_\bbZ\,x_\mu$ in $\Gr_{G_\bbZ}$ has a cellular paving over $\bbZ$.
\end{thm}

Note that the second statement gives an alternate proof of a recent result of Cass--van den Hove--Scholbach, namely \cite[Theorem~1.2]{CvdHS+}.

\begin{proof}
The proofs of the results over fields can be directly imported to the context over $\bbZ$, using in particular Corollary~\ref{BP_comp_Z}, Lemma~\ref{Omega_Z}, Lemma~\ref{BN_pair_lem_Z}, Equation (\ref{cl_rel_Z}), and Lemma~\ref{key_fiber_lem_Z}.  With these tools in hand, the proof over fields works over $\bbZ$ with no changes.  Note that we do not really need the language of retractions at any point in the proof: Every fact justified using retractions is equivalent to a purely group-theoretic statement.  See for example Remarks~\ref{A1_rem} and~\ref{Gm_rem}.
\end{proof}

\section{Errata for \texorpdfstring{\cite{dHL18}}{[dCHL18]}} \label{Errata}

We take this opportunity to point out a few minor mistakes in \cite{dHL18}.  In \cite[Proposition~3.10.2]{dHL18}, we stated that all Schubert varieties  $X_\calP(w)$ in partial affine flag varieties $\Flag_\calP$ are normal.  This is true for classical Schubert varieties (those contained in $G/P$ for a parabolic subgroup $P$ in $G$) but is false in general.  Pappas and Rapoport proved in \cite{PR08} that normality does hold for all affine Schubert varieties attached to $G$ over a field $k$, as long as the characteristic of $k$ is coprime to the order of the Borovoi fundamental group $\pi_1(G_{\drv})$ (see \cite{Bo98}). However, when $\cha(k)$ divides $|\pi_1(G_{\drv})|$, it is proved in \cite[Theorem~2.5]{HLR} that {\em most} Schubert varieties in $\Flag_\calP$ are {\em not normal}.

The normality of Schubert varieties is invoked in \cite[Corollary~4.1.4]{dHL18} to prove that the convolution space $X_\calP(w_\bullet)$ is normal. This also fails in general but is true when $\cha(k) \nmid |\pi_1(G_{\drv})|$.  In addition, normality of Schubert varieties is used in one of the proofs in \cite{dHL18} that the fibers of convolution morphisms $X_\calP(w_\bullet) \rightarrow X_\calP(w_*)$ are geometrically connected.  More precisely, a normality hypothesis plays a role in \cite[Proposition~4.4.4]{dHL18}, which in turn is used to prove the geometric connectedness of the fibers in \cite[Corollary~4.4.5]{dHL18}.  This proof is not valid in general but is valid, again, under the hypothesis $\cha(k) \nmid |\pi_1(G_{\drv})|$.  Fortunately, in \cite[Theorem~2.2.2]{dHL18}, another proof of the geometric connectedness of the fibers is given, which does not rely on any normality of Schubert varieties. 

Furthermore, the polynomials $F_{p,v}(q)$ appearing in \cite[Equation~(2.1)]{dHL18} were defined incorrectly as the Poincar\'{e} polynomials of the fibers $p^{-1}(v\calB)$.  They are rather the functions $$F_{p,v}(q) = \tr\big({\rm Frob}_q ~,~ \sum_i (-1)^i \, {\rm H}^i(p^{-1}(v\calB) \,,\, {\mathcal IC}_{X_{\calB}(w_\bullet)})\big).$$  
The fact that $F_{p,v}(q) \in \mathbb Z_{\geq 0}[q]$ is not \textit{a priori} obvious, but it follows from \cite[Equation~(2.1)]{dHL18}.



\begin{thebibliography}{CvdHS22+++}

\bibitem[Bor98]{Bo98} M.~Borovoi, {\em Abelian Galois cohomology of reductive groups}, Mem.\ Amer.\ Soc.\ Math.\ \textbf{132} (1998), no.~626, \doi{10.1090/memo/0626}.
  
\bibitem[Bou68]{Bou} N.~Bourbaki, {\em \'El\'ements de Math\'ematique. Fasc.\ XXXIV. Groupes et alg\'ebres de Lie. Chapitre IV: Groupes de Coxeter et syst\'emes de Tits. Chapitre V: Groupes engendr\'es par des r\'eflexions. Chapitre VI: Syst\'emes de racines}, Actualit\'es Sci.\ Indust.\ vol.~1337, Hermann, Paris, 1968.
 
 
\bibitem[BT72]{BT72} F.~Bruhat and J.~Tits, {\em Groupes r\'eductifs sur un corps local I. Donn\'ees radicielles valu\'ees}, Inst.\ Hautes \'Etudes Sci.\ Publ.\ Math.\ \textbf{41} (1972), no.~1, 5--251, 
\doi{10.1007/BF02715544}.
 
\bibitem[BT84]{BT84} \bysame, {\em Groupes r\'eductifs sur un corps local II. Sch\'emas en groupes. Existence d'une donn\'ee radicielle valu\'ee}, Inst.\ Hautes \'Etudes Sci.\ Publ.\ Math.\ \textbf{60} (1984), 197--376,
  \doi{10.1007/BF02700560}.
    
\bibitem[CvdHS22]{CvdHS+} R.~Cass, T.~van den Hove, and J.~Scholbach, {\em The geometric Satake equivalence for integral motives}, preprint \arXiv{2211.04832} (2022).
 
\bibitem[dCHL18]{dHL18}
M.\,A.~de~Cataldo, T.\,J.~Haines, and L.~Li, {\em Frobenius semisimplicity for convolution morphisms}, Math.~Z.\ {\bf 289} (2018), no.~1-2, 119--169, 
\doi{10.1007/s00209-017-1946-4}.
 
\bibitem[\v{C}es22]{Ces22} K.~\v{C}esnavi\v{c}ius, {\em Torsors on the complement of a smooth divisor}, Camb.~J.\ Math.\ 12 (2024), no.~4, 903–-936, \doi{10.4310/cjm.241128010204}.
 
\bibitem[Con14]{Co14} B.~Conrad, {\em Reductive group schemes}, In: \emph{Autour des sch\'emas en groupes. Vol.~I}, pp.~93--444, Panor.\ Synth\'eses \textbf{42-43}, Soc.\ Math.\ France, Paris, 2014.
  
\bibitem[Fal03]{Fal03} G.~Faltings, {\em Algebraic loop groups and moduli spaces of bundles}, J.\ Eur.\ Math.\ Soc.\ (JEMS) {\bf 5} (2003), no.\,1, 41--68,
\doi{10.1007/s10097-002-0045-x}.
 
\bibitem[Gai01]{Ga01} D.~Gaitsgory, {\em Construction of central elements in the affine Hecke algebra via nearby cycles}, Invent.\ Math.~\textbf{144} (2001), no.\,2, 253--280, 
\doi{10.1007/s002220100122}.
 
\bibitem[Hai05]{H05} T.\,J.~Haines, {\em A proof of the Kazhdan-Lusztig purity theorem via the decomposition theorem of BBD}, note, around 2005, available at \url{www.math.umd.edu/~tjh}.
 
\bibitem[Hai06]{Hai06} \bysame, {\em Equidimensionality of convolution morphisms and applications to saturation problems}, Adv.\ Math.\ {\bf 207} (2006), no.~1, 297--327, 
\doi{10.1016/j.aim.2005.11.014}.
 
\bibitem[HKM12]{HKM} T.\,J.~Haines, M.~Kapovich, and J.\,J.~Millson, {\em Ideal triangles in Euclidean buildings and branching to Levi subgroups}, J.~Algebra {\bf 361} (2012), 41--78, 
\doi{10.1016/j.jalgebra.2012.04.001}.
 
\bibitem[HKP10]{HKP} T.\,J.~Haines, R.~Kottwitz, and A.~Prasad, {\em Iwahori-Hecke algebras}, J.~Ramanujan Math.\ Soc.~{\bf 25}, no.~2 (2010), 113--145.
 
\bibitem[HLR18]{HLR} T.\,J.~Haines, J.~Louren\c{c}o, and T.~Richarz, {\em On normality of Schubert varieties: remaining cases in positive characteristic}, Ann.\ Sci.\ \'{E}c.\ Norm.\ Sup\'er.\ (4) {\bf 57} (2024), no.~3, 895–-959.
  
\bibitem[HR08]{HR08} T.~Haines and M.~Rapoport, {\em Appendix: On parahoric subgroups}, Adv.\,Math.\,{\bf 219} (2008), no.~1, 188--198, 
\doi{10.1016/j.aim.2008.04.020}.
 
\bibitem[HR20]{HR20} 
T.\,J.~Haines and T.~Richarz, {\em The test function conjecture for local models of Weil-restricted groups}, Compos.\ Math.\ {\bf 156} (2020), no.~7,  1348--1404, \doi{10.1112/S0010437X20007162}.
 
\bibitem[HR23]{HaRi2} \bysame, {\em Normality and Cohen-Macaulayness of parahoric local models}, J.\ Eur.\ Math.\ Soc.\ {\bf 25} (2022), no.~2, 703--729, \doi{10.4171/jems/1192}.
 
\bibitem[Hum90]{Hum} J.\,E.~Humphreys, {\em Reflection groups and Coxeter groups}, Cambridge Stud.\ Adv.\ Math.\ {\bf 29}, Cambridge Univ.\ Press, Cambridge, 1990, \doi{10.1017/CBO9780511623646}. 
 
  
\bibitem[KLM]{KLM} M.~Kapovich, B.~Leeb, and J.~Millson, {\em The generalized triangle inequalities in symmetric spaces and buildings with applications to algebra}, Mem.\ Amer.\ Math.\ Soc.\ {\bf 192} (2008), no.~896, \doi{10.1090/memo/0896}.

\bibitem[Lou23]{Lou23} J.~Louren\c{c}o, {\em Grassmanniennes affines tordues sur les entiers}, Forum Math.\ Sigma {\bf 11} (2023), no.~e12, \doi{10.1017/fms.2023.4}.
  
\bibitem[MV07]{MV07} I.~Mirkovi\'c and K.~Vilonen, {\em Geometric Langlands duality and representations of algebraic groups over commutative rings}, Ann.\ Math. \textbf{166} (2007), no.~1, 95--143, 
\doi{10.4007/annals.2007.166.95}.
 
\bibitem[PR08]{PR08} G.~Pappas and M.~Rapoport, {\em Twisted loop groups and their affine flag varieties}, Adv.\ Math.\ \textbf{219} (2008), no.~1, 118--198, 
\doi{10.1016/j.aim.2008.04.006}.
 
\bibitem[PZ13]{PZ13} G.~Pappas and X.~Zhu, {\em Local models of Shimura varieties and a conjecture of Kottwitz}, Invent.\ Math.\ \textbf{194} (2013), no.~1, 147--254, 
\doi{10.1007/s00222-012-0442-z}.
 
\bibitem[Par06]{Par06} J.~Parkinson, {\em Buildings and Hecke algebras}, J.\,Alg.\,{\bf 297} (2006), 1--49. 
\doi{10.1016/j.jalgebra.2005.08.036}.
 
\bibitem[Ric14]{Ri14} T.~Richarz, {\em A new approach to the geometric Satake equivalence}, Doc.\ Math.\ \textbf{19} (2014) 209--246, 
\doi{doi.org/10.4171/dm/445}.
 
\bibitem[Ric16]{Ri16b} \bysame, {\em Affine Grassmannians and Geometric Satake equivalences}, Int.\ Math.\ Res.\ Not.\ (2016), no.~12, 3717--3767,
\doi{10.1093/imrn/rnv226}.
 
\bibitem[RS20]{RS20} T.~Richarz and J.~Scholbach, {\em The intersection motive of the moduli stack of shtukas}, Forum Math.\ Sigma {\bf 8} (2020), no.~e8, 
\doi{10.1017/fms.2019.32}.
 
\bibitem[Schw06]{Schw06} C.~Schwer, {\em Galleries, Hall-Littlewood polynomials, and structure constants of the spherical Hecke algebra}, Int.\ Math.\ Res.\ Not.\ 2006, Art.\ ID 75395,
  \doi{10.1155/IMRN/2006/75395}.
  
\bibitem[Spa76]{Sp} N.~Spaltenstein, {\em The fixed point set of a unipotent transformation on the flag manifold}, Nederl.\ Akad.\ Wetensch.\ Proc.\ Ser.~A {\bf 79}, Indag.\ Math {\bf 38} (1976), no.~5, 452--456,  
  \doi{10.1016/S1385-7258(76)80008-X}.
 
\bibitem[Spr98]{Spr} T.\,A.~Springer, {\em Linear Algebraic Groups}, 2nd ed., Progr.\ Math.\ vol.~{\bf 9}, Birkh\"{a}user Boston, Inc., Boston, MA, 1998, \doi{10.1007/978-0-8176-4840-4}.
 
\bibitem[Sta24]{StaPro} The Stacks Project Authors, \emph{The Stacks Project}, \url{http://stacks.math.columbia.edu}, 2024.
  
\end{thebibliography}
\end{document}